\documentclass[fleqn,final]{zzz}
\usepackage{mathrsfs}
\usepackage{showkeys} 
\usepackage{color}
\usepackage{tikz}
\usepackage{amsmath,amsthm,amsbsy,amssymb}
\usepackage{graphicx}
\usepackage{rotating}
\usepackage{fixmath}
\usepackage[pdftex]{hyperref}
\usepackage{soul}
\usepackage{comment}

\hypersetup{
 colorlinks = true,
 urlcolor = blue,
 linkcolor = red,
 citecolor = green,
}

\sethlcolor{black}

\newcommand{\Whyp}[5]{\,\mbox{}_{#1}W_{#2}\!\left({#3};{#4};{#5}\right)}
\newcommand{\Wpsi}[5]{\,\mbox{}_{#1}\Psi_{#2}\!\left({#3};{#4};{#5}\right)}
\newcommand{\qhyp}[5]{\,\mbox{}_{#1}\phi_{#2}\!\left(
\genfrac{}{}{0pt}{}{#3}{#4};#5\right)}
\newcommand{\qphyp}[6]{\,\sideset{_{#1}^{\phantom{\mid}}}{_{#2}^{#3}}
{\mathop{\phi}}\!\left(\genfrac{}{}{0pt}{}{#4}{#5};#6\right)}
\newcommand{\qppsi}[6]{\,\sideset{_{#1}^{\phantom{\mid}}}{_{#2}^{#3}}
{\mathop{\psi}}\!\left(\genfrac{}{}{0pt}{}{#4}{#5};#6\right)}

\newcommand{\hyp}[5]{\,\mbox{}_{#1}F_{#2}\!\left(
 \genfrac{}{}{0pt}{}{#3}{#4};#5\right)}
 \newcommand{\Hhyp}[5]{\,\mbox{}_{#1}H_{#2}\!\left(
 \genfrac{}{}{0pt}{}{#3}{#4};#5\right)}
\newcommand{\nqphyp}[3]{\sideset{_{#1}^{\phantom{\mid}}}{_{#2}^{#3}}{\mathop{\phi}}}
\newcommand{\nlqphyp}[3]{\sideset{_{#1}}{_{#2}^{#3}}{\mathop{\phi}}}
\newcommand{\nqppsi}[3]{\sideset{_{#1}^{\phantom{\mid}}}{_{#2}^{#3}}{\mathop{\psi}}}
\newcommand{\nlqppsi}[3]{\sideset{_{#1}}{_{#2}^{#3}}{\mathop{\psi}}}

\newcommand{\II}[1]{\overset{\raisebox{0.5ex}{$#1$}}{\raisebox{-0.55cm}
{\resizebox{0.3cm}{!}{\scalebox{0.4}[1.0]{\mbox{$\mathbb{I}$}}}}}}

\newtheorem{thm}{Theorem}[section]
\newtheorem{cor}[thm]{Corollary}

\newtheorem{rem}[thm]{Remark}
\newtheorem{lem}[thm]{Lemma}
\newtheorem{defn}[thm]{Definition}

\makeatletter
\def\eqnarray{\stepcounter{equation}\let\@currentlabel=\theequation
\global\@eqnswtrue
\tabskip\@centering\let\\=\@eqncr
$$\halign to \displaywidth\bgroup\hfil\global\@eqcnt\z@
$\displaystyle\tabskip\z@{##}$&\global\@eqcnt\@ne
\hfil$\displaystyle{{}##{}}$\hfil
&\global\@eqcnt\tw@ $\displaystyle{##}$\hfil
\tabskip\@centering&\llap{##}\tabskip\z@\cr}

\def\endeqnarray{\@@eqncr\egroup
\global\advance\c@equation\m@ne$$\global\@ignoretrue}

\def\@yeqncr{\@ifnextchar [{\@xeqncr}{\@xeqncr[5pt]}}
\makeatother

\newcommand{\Z}{\mathbb{Z}} 
\newcommand{\N}{\mathbb{N}} 

\newcommand{\CC}{{{\mathbb C}}}
\newcommand{\CCast}{{{\mathbb C}^\ast}}

\newcommand{\qpsi}[5]{\,\mbox{}_{#1}\psi_{#2}\!\left(
\genfrac{}{}{0pt}{}{#3}{#4};#5\right)}
\newcommand{\qpWhyp}[6]{\,\sideset{_{#1}^{\phantom{\mid}}}{_{#2}^{#3}}
{\mathop{W}}\!\left({#4};{#5};#6\right)}
\newcommand{\qpWpsi}[6]{\,\sideset{_{#1}^{\phantom{\mid}}}{_{#2}^{#3}}
{\mathop{\Psi}}\!\left({#4};{#5};#6\right)}

\newcommand{\vt}{{\vartheta}}

\newcommand{\midtilde}{\raisebox{-0.25\baselineskip}{\textasciitilde}}

\usepackage{scalerel,url}
\makeatletter
 \AtBeginDocument{
 \check@mathfonts
 \fontdimen13\textfont2=3.5pt
 \fontdimen14\textfont2=3.5pt
 \fontdimen16\textfont2=2.5pt
 \fontdimen17\textfont2=2.5pt
 }
\makeatother

\allowdisplaybreaks

\begin{document}

\renewcommand{\PaperNumber}{***}

\FirstPageHeading

\ShortArticleName{Transformations and summations for bilateral basic hypergeometric series}

\ArticleName{Transformations and summations for bilateral\\basic hypergeometric series}

\Author{Howard S. Cohl$\,^{\ast
}$ and Michael J. Schlosser$\,^{\ast\ast
}$ 
}

\AuthorNameForHeading{H.~S.~Cohl}
\Address{$^\ast$ Applied and Computational 
Mathematics Division, National Institute of Standards 
and Tech\-no\-lo\-gy, Mission Viejo, CA 92694, USA
\URLaddressD{
\href{http://www.nist.gov/itl/math/msg/howard-s-cohl.cfm}
{http://www.nist.gov/itl/math/msg/howard-s-cohl.cfm}
}
} 
\EmailD{howard.cohl@nist.gov} 

\AuthorNameForHeading{H.~S.~Cohl, M.~J.~Schlosser}
\Address{$^{\ast\ast}$ Fakult\"at f\"ur Mathematik,
Universit\"at Wien,
Oskar-Morgenstern-Platz 1, Vienna,
Austria
\URLaddressD{
\href{https://www.mat.univie.ac.at/~schlosse/}
{https://www.mat.univie.ac.at/\midtilde{}schlosse/}
}
} 
\EmailD{michael.schlosser@univie.ac.at} 


\ArticleDates{Received \today~in final form ????; Published online ????}

\Abstract{%
We review and derive transformation and summation formulas for bilateral basic hypergeometric series.
Our study focuses on consequences of certain bilateral extensions of two important results by Bailey,
namely a transformation for very-well-poised $_8W_7$ series in terms of two balanced $_4\phi_3$ series,
and a transformation connecting three $_8W_7$ series.
Two rather recently discovered transformations of bilateral basic very-well-poised ${}_8\Psi_8$,
one by Zhang and Zhang, the other by Wei and Yu, serve as the starting point of our investigations.
From these transformations we work out interesting special cases that were not considered before,
including explicit bilateral quadratic and cubic summations. We further
explicitly record noteworthy lower-level transformations derived by taking suitable limits and
deduce more transformations by exploiting the symmetry of the parameters in the series.}

\Classification{33D15; 33D50}

\begin{flushright}
\begin{minipage}{60mm}
\it Dedicated to Mourad E.~H.~Ismail\\on the occasion of his 80\textsuperscript{th} birthday.
\end{minipage}
\end{flushright}

\section{Introduction}
The theory of basic hypergeometric series \cite{GaspRah} is rich in identities.
This theory arose out of an interest to extend classical results for hypergeometric series
to that for the \textit{basic} (or $q$-)case. The first systematic development
of the theory was undertaken by Heine in the 1840s, but only some decades
later, in the first half of the 20th century, its development attracted broad interest,
with researchers such as Jackson, Bailey, Watson, Slater, and others, 
being leading pioneers.

While numerous identities (summation and transformation formulas) for unilateral
basic hypergeometric series exist, not so many identities are known for \textit{bilateral}
basic hypergeometric series, the list of closed form summations essentially being
limited (in one variable, i.e., for single series) to Ramanujan's $_1\psi_1$ summation
and Bailey's very-well-poised $_6\psi_6$ summation and their special cases
(such as Jacobi's triple product identity, and the quintuple product identity),
apart from some other summations (such as those of $q$-Karlsson--Minton type
which even hold for a general number of parameters, or, on the contrary,
summations for some highly specialized bilateral basic hypergeometric series,
such as those of Rogers--Ramanujan type \cite{Schlosserbilateral}).

This paper is devoted to the study of identities for bilateral basic hypergeometric series
(in one variable), focusing on identities containing series up to the level of
$_8\psi_8$ basic hypergeometric series. (For identities on the level of $_{10}\psi_{10}$ series we merely review the general ones considered by Slater. However, they do not appear to be bilateral extensions of Bailey's 4-term very-well-poised ${}_{10}W_9$ transformations or of relevant special cases of them, such as Bailey's 3-term very-well-poised $_8\phi_7$ transformations.) The identities we feature are all deduced
from transformation formulas for $_8\psi_8$ series obtained by
Zhang and Zhang~\cite{ZhangZhang2007} (we correct their main result)
and by Wei and Yu~\cite{WeiYu2021}.
On one hand, we hope that the compilation of identities we give, having everything at one place, is useful. On the other hand,
we list some particularly nice consequences that were not considered
explicitly before, such as an interesting cubic summation, or
some quadratic summations; these are derived by specializing the
$_8\psi_8$ transformations and combining them with other known results.


The remainder of this introductory section is devoted to some preliminaries
regarding notation. In Section~\ref{sec:uni} we recall some important identities
involving unilateral basic hypergeometric series that we make use of. In Section~\ref{sec:bil} we turn to bilateral basic hypergeometric series
with a focus on various transformations satisfied by them. In Theorem~\ref{thm22}, we have presented a corrected version of the Zhang--Zhang transformation of a very-well-poised ${}_8\psi_8$ \cite[Theorem 6]{ZhangZhang2007} which is a bilateral extension of Bailey's transformation \cite[(17.9.16)]{NIST:DLMF} of a nonterminating very-well-poised ${}_8W_7$ in terms of a sum of two balanced ${}_4\phi_3$ series. In the new Theorem~\ref{thm337H7}, we give a $q\to 1^{-}$ limit of the Zhang--Zhang transformation which provides a transformation of a very-well-poised ${}_7H_7$ in terms of a ${}_4H_4$ and two balanced ${}_4F_3$'s. By starting with Theorem  \ref{thm22}, we obtain in Corollary~\ref{corbcubic} a new cubic summation formula involving a very-well-poised ${}_8\psi_8$ series and a ${}_2\phi_1$ series with base $q^3$.
The specialization of Corollary~\ref{corbcubic} to the unilateral case
(a cubic summation involving a very-well-poised ${}_8\phi_7$ series and a ${}_2\phi_1$ series with base $q^3$)
already appears to be new and is given as Corollary~\ref{corcubic}. In Corollary~\ref{cor38} we present a new transformation of a very-well-poised ${}_8W_7$ in terms a ${}_4\psi_4$ and two balanced ${}_4\phi_3$'s. In Theorem~\ref{thmwy} we revisit the Wei--Yu transformation of a very-well-poised ${}_8\psi_8$ in terms another very-well-poised ${}_8\psi_8$ and a very-well-poised ${}_8W_7$ \cite[Theorem 3]{WeiYu2021}. In Corollaries \ref{cor311a}, \ref{cor311}, \ref{cor312} we present new nonterminating summation formulas for very-well-poised ${}_8\psi_8$'s. 
Corollaries \ref{cor311}, \ref{cor312} are bilateral extensions of the nonterminating very-well-poised ${}_8W_7$ summations  \cite[(II.16), (II.18)]{GaspRah}. 
By equating the Zhang--Zhang and Wei--Yu transformations, we give a new transformation of a very-well-poised ${}_8\psi_8$, a ${}_4\psi_4$, an ${}_8W_7$, and two balanced ${}_4\phi_3$'s. Note that in \cite{WeiYu2021} the authors only realized that their Theorem~\ref{thmwy} extended Bailey's ${}_6\Psi_6$ summation (see Remark \ref{rem310a}). Other than this, all consequences of the Wei--Yu transformation which appear in this manuscript appear to be new.

In Theorem~\ref{thm21}, we derive a sequence of transformation formulas by starting with Bailey's transformation of a very-well-poised ${}_8W_7$ in terms of a sum of two balanced ${}_4\phi_3$'s. This interesting sequence of transformations was studied in \cite{vandeBultRains09} (see in particular \cite[Figure 2]{vandeBultRains09}). van de Bult and Rains (2009) start from Bailey's four-term transformation formula for nonterminating ${}_{10}W_9$. Starting with Bailey's transformation of a very-well-poised
${}_8W_7$ in terms of a sum of two balanced ${}_4\phi_3$'s, the first limit gives a transformation of a ${}_7W_6^1$ series to a special ${}_3\phi_2$ series \eqref{W761tran}. The second limit gives the transformation of a ${}_6W_5^2$ to a ${}_2\phi_1$
with an arbitrary argument and general parameters.
The extension of this transformation to bilateral basic hypergeometric series
was discovered by Bailey \eqref{Bailey}; it relates a ${}_6\Psi_6^2$ to a ${}_2\psi_2$.
One of the goals of this paper is to study a scheme of transformations for bilateral basic hypergeometric series
that is analogous to that for the unilateral transformations of Bailey-type, where the elements in the scheme 
are connected by taking limits.

Consequences of the featured transformation formulas are given in Sections~\ref{sec:lim}
and \ref{sec:imp}. Since the corrected Zhang--Zhang result provides a bilateral extension of Bailey's transformation \eqref{Baileytran}, we utilize that result in Section \ref{sec:lim} to deduce bilateral lower-level
transformations by suitably taking limits.
In Section \ref{sec:imp}, we derive further transformation formulas, not by taking limits but by starting
with the transformation formulas derived in Section \ref{sec:lim} and exploiting the symmetry in the parameters.
While the transformation formulas in Sections \ref{sec:lim} and \ref{sec:imp} are all routine consequences
of more general results, we record them explicitly as we believe that it is useful to have them available in one place.
As far as we are aware, they have not
been stated explicitly before (with exceptions, such as the very useful Corollary~\ref{cor63b}),
and appear for the first time in this paper.

\subsection{Preliminaries}

\begin{defn}\label{def:1.1}
We adopt the following conventions for succinctly 
writing elements of multisets. To indicate 
sequential positive and negative 
elements, we write
\[
\pm a:=\{a,-a\}.
\]
\noindent We also adopt an analogous notation
\[
z^{\pm}:=\{z,z^{-1}\}.
\]
\end{defn}
\noindent We adopt the following set 
notations: $\mathbb N_0:=\{0\}\cup
\mathbb N=\{0, 1, 2,\ldots\}$, and we 
use the sets $\mathbb Z$, $\mathbb R$, 
$\mathbb C$ which represent 
the integers, real numbers, and 
complex numbers respectively, 
$\CCast:=\CC\setminus\{0\}$.  
\noindent 
Define the set
$q^{\mathbb Z}:=\{q^k:k\in\Z\}.$
The 
$q$-shifted factorial is given by 
\begin{equation}
(a;q)_n:=(1-a)(1-qa)\cdots(1-q^{n-1}a),\qquad n\in\mathbb N_0,
\end{equation}
and one further defines
\begin{eqnarray}
&&\hspace{-11.3cm}(a;q)_\infty:=\prod_{n=0}^\infty 
(1-aq^{n}),\label{poch.id:2}
\end{eqnarray}
where $|q|<1$. Using this infinite product, one can extend
the $q$-shifted factorial to arbitrary complex index $b$,
by defining
\begin{equation}
(a;q)_b:=\frac{(a;q)_\infty}{(q^ba;q)_\infty}.
\end{equation}
In particular, when dealing with bilateral series we will frequently
use the $q$-shifted factorial with the index $b$ being a negative integer,
say, $b=-n$ where $n\in\N$, in which case we see that
\begin{equation}
(a;q)_{-n}=q^{\binom{n+1}2}(-a)^{-n}\Big(\frac qa;q\Big)_n^{-1}
\end{equation}
(an identity which is actually valid for any $n\in\Z$).
We will also use the common 
notational product convention
\begin{eqnarray}
&&\hspace{-8cm}(a_1, \ldots, a_k;q)_b:=
(a_1;q)_b\cdots(a_k;q)_b,\nonumber
\end{eqnarray}
where $b\in{\mathbb C}\cup\{\infty\}$.
The {\it theta function} $\vartheta(z;q)$ (sometimes referred to as a modified theta function 
\cite[(11.2.1)]{GaspRah})
is defined
by 
Jacobi's triple product identity and is
given by {\cite[(1.6.1)]{GaspRah}} (see also \cite[(2.3)]{Koornwinder2014})
\begin{equation}
\vartheta(z;q):=
(z,\tfrac{q}{z};q)_\infty=\frac{1}{(q;q)_\infty}\sum_{n=-\infty}^\infty q^{\binom{n}{2}}(-z)^n,
\label{tfdef}
\end{equation}
where $z\ne 0$, $|q|<1$. Note that $\vartheta(q^n;q)=0$ if
$n\in\Z$.
We will use the compact product notation
\begin{equation*}
\vartheta(z_1,\ldots,z_k;q):=\vartheta(z_1;q)\cdots\vartheta(z_k;q).
\end{equation*}

\begin{lem}One has the following modified theta function addition formula \cite[Exercise 2.16(i)]{GaspRah}
\begin{equation}
\vt\Big(e,\frac{e}{c},\frac{qa}{d},\frac{qc}{ad};q\Big)
-\vt\Big(d,\frac{d}{c},\frac{qa}{e},\frac{qc}{ae};q\Big)=
\frac{d}{c}\,\vt\Big(a,\frac{c}{a},\frac{e}{d},\frac{de}{c};q\Big).
\label{tfa}
\end{equation}
\end{lem}
\noindent The $q$-shifted factorial 
also has the 
following useful property
\cite[(1.8.10)]
{Koekoeketal}:
\begin{eqnarray}
\label{poch.id:3} 
&&\hspace{-10.1cm}
(a;q^{-1})_n=q^{-\binom{n}{2}}(-a)^n(a^{-1};q)_n.
\end{eqnarray}
Note the equivalent representation of \eqref{poch.id:3} 
which is very useful 
for obtaining limits of basic hypergeometric and bilateral basic hypergeometric series, 
\begin{equation*}
a^n\left(\frac{x}{a};q\right)_n=
q^{\binom{n}{2}}(-x)^n\left(\frac{a}{x};q^{-1}\right)_n,
\end{equation*}
therefore
\begin{equation}
\lim_{a\to0}\,a^n\left(\frac{x}{a};q\right)_n=
\lim_{b\to\infty}\,\frac{1}{b^n}\left(xb;q\right)_n=
q^{\binom{n}{2}}(-x)^n.
\label{critlim}
\end{equation}

\noindent Furthermore, one has the following
identities
 \cite[(5), (6)]{CohlCostasSantos23},
\begin{eqnarray}
&&\hspace{-11cm}\label{sq}
(a^2;q)_\infty=(\pm a,\pm q^\frac12 a;q)_\infty,\\
&&\hspace{-11.cm}\label{aiden}
\frac{(a,\frac{q}{a};q)_\infty}{(qa,\frac{1}{a};q)_\infty}=
\frac{\vt(a;q)}{\vt(qa;q)}
=-a.
\end{eqnarray}
Finally, we will use the following {\it capital I}
notation
\begin{eqnarray}
&&\hspace{-2.4cm}\II{x_1;x_2,\dots,x_k} f(x_1,\ldots,x_k):=
\sum_{j=1}^kf(x_j,x_2,x_3,\ldots, x_{j-1},x_1,x_{j+1},\ldots,x_{k-1},x_k)
\end{eqnarray}
($f(x_1,\ldots,x_k)$ being any function depending on $x_1,\ldots,x_k$),
appearing e.g.,~in \eqref{psi101010Psi10}, \eqref{psi101010W9}.
This means that the whole expression after the large {\it I} symbol on the left
is taken as a summand and is then
additively repeated, with $x_2$ switched with $x_1$,
then with $x_3$ switched with $x_1$, etc.,
and finally with $x_k$ switched with $x_1$,
the sum having $k$ terms in total.
We refer to this notation as \textit{idem} notation (which we employ instead of the notation $\operatorname{idem} (x_1;x_2,\ldots,x_k) f(x_1,\ldots,x_k)$) frequently appearing in the literature on bilateral (basic) hypergeometric series, see e.g., Sears (1951) \cite[p.~481]{Sears1951}, Slater (1966) \cite[p.~192]{Slater66}, \cite[(4.5.2)]{GaspRah}).

\section{Unilateral basic hypergeometric series transformations}
\label{sec:uni}

For (ordinary) hypergeometric series (sometimes called \textit{generalized hypergeometric series},
and denoted as $_rF_s$ series, to distinguish them from the classical hypergeometric $_2F_1$ function,
extensively studied by Euler, Gauss, Riemann, Kummer, etc.),
we refer the reader to the classic textbook by Slater \cite{Slater66}.

The basic hypergeometric series (cf.\ \cite{GaspRah}), which we 
frequently use, is defined for
$q,z\in\CCast$ such that $|q|<1$, $s,r\in\mathbb N_0$, 
$b_j\not\in q^{-\N_0}$, 
$j=1,...,s$, as
\cite[(1.10.1)]{Koekoeketal}
\begin{equation}
\qhyp{r}{s}{a_1,...,a_r}
{b_1,...,b_s}
{q,z}
:=\sum_{k=0}^\infty
\frac{(a_1,...,a_r;q)_k}
{(q,b_1,...,b_s;q)_k}
\left((-1)^kq^{\binom k2}\right)^{s-r+1}
z^k,
\label{2.11}
\end{equation}
whenever the series converges.
This nonterminating infinite series \eqref{2.11} is entire if and only if $s-r+1 > 0$, is convergent for $s=r-1$ if $|z|<1$ and is divergent for $s-r+1<0$.
We sometimes specifically refer to the series in \eqref{2.11} as a \textit{unilateral} basic hypergeometric series,
to distinguish it from a \textit{bilateral} basic hypergeometric series that we consider in Section~\ref{sec:bil}.

Note that we refer to a basic hypergeometric
series as {\it balanced} if
$qa_1\cdots a_r=b_1\cdots b_s$.
Further, define the nonterminating very-well-poised 
basic hypergeometric series
${}_{r+1}W_r$ \cite[(2.1.11)]{GaspRah}
\begin{equation}
\label{rWr}
{}_{r+1}W_r(a;a_4,\ldots,a_{r+1};q,z)
:=\qhyp{r+1}{r}{a,\pm q\sqrt{a},a_4,\ldots,a_{r+1}}
{\pm \sqrt{a},\frac{qa}{a_4},\ldots,\frac{qa}{a_{r+1}}}{q,z},
\end{equation} 
where $\sqrt{a},\frac{qa}{a_4},\ldots,\frac{qa}{a_{r+1}}\not\in q^{-\N_0}$. 
We will use the following notation 
$\nqphyp{r}{s}{m}$, $m\in\N_0$
(originally due to van de Bult and Rains
\cite[p.~4]{vandeBultRains09}), 
for basic hypergeometric series when some denominator elements are equal to zero.
Let $p\in\mathbb N_0$.
Then define, for multisets ${\bf a}:=\{a_1,\ldots,a_{r+1}\}$ and
${\bf b}:=\{b_1,\ldots,b_s\}$,
\begin{equation}\label{topzero} 
{}_{r+1}\phi_s^{-p}\left(\begin{array}{c}{\bf a}\\
{\bf b}\end{array};q,z
\right)
:=
\qhyp{r+p+1}{s}
{{\bf a},\underbrace{0,\ldots,0}_{p}
}
{\bf b}{q, z},
\end{equation}
\begin{equation}\label{botzero}
{}_{r+1}\phi_s^{\,p}\left(\begin{array}{c}{\bf a}\\
{\bf b}\end{array};q,z\right)
:=
\qhyp{r+1}{s+p}{\bf a}
{{\bf b},
\overbrace{0,\ldots,0}^{p}}{q,z},
\end{equation}
where $b_1,\ldots,b_s\not
\in q^{-\N_0}\cup\{0\}$, and
$
\nlqphyp{r+1}{s}{0}
={}_{r+1}\phi_{s}
.$
The nonterminating basic hypergeometric series 
$\nlqphyp{r+1}{s}{m}
({\bf a};{\bf b};q,z)$ is well-defined for $s-r+m\ge 0$.
In particular, $\nlqphyp{r+1}{s}{m}$
is an entire function of $z$ for $s-r+m>0$, convergent for $|z|<1$ for $s-r+m=0$
and divergent if $s-r+m<0$ {unless the series is terminating}.
Note that we will move interchangeably between the
van de Bult and Rains notation and the alternative
notation with vanishing numerator and denominator elements
which are used on the right-hand side of 
\eqref{botzero}.

Further, consider the nonterminating very-well-poised 
basic hypergeometric series with vanishing denominator elements
${}_{r+1}W_r^p$ \cite[p.~4]{vandeBultRains09}
\begin{equation}
\label{rpWr}
{}_{r+1}W_r^p(a;a_4,\ldots,a_{r+1};q,z)
:=\qphyp{r+1}{r}{p}{a,\pm q\sqrt{a},a_4,\ldots,a_{r+1}}
{\pm \sqrt{a},\frac{qa}{a_4},\ldots,\frac{qa}{a_{r+1}}}{q,z},
\end{equation} 
where $\sqrt{a},\frac{qa}{a_4},\ldots,\frac{qa}{a_{r+1}}\not\in q^{-\N_0}$. Recognize that ${}_{r+1}W_r^0={}_{r+1}W_r$. Note that in the case where $p<0$, then the nonterminating basic hypergeometric series ${}_{r+1}W_r^p$ is divergent since the corresponding nonterminating basic hypergeometric series ${}_r\phi_s$ \eqref{2.11} has $s-r+1<0$ (and thus, the series contains negative
quadratic powers of $q$, responsible for its divergence).
Therefore we only consider the nonterminating case ${}_{r+1}W_r^p$ with $p\ge 0$.

\subsection{Nonterminating transformations}
\label{nontermbhssec}
One has the following nonterminating transformation between a ${}_2\phi_2$ 
and a ${}_2\phi_1$ (cf.~\cite[(III.4)]{GaspRah})
\begin{equation}
\qhyp22{a,b}{c,\frac{abz}{c}}{q,z}=
\frac{(\frac{bz}{c};q)_\infty}{(\frac{abz}{c};q)_\infty}
\qhyp21{a,\frac{c}{b}}{c}{q,\frac{bz}{c}},
\label{rel2122}
\end{equation}
valid for $|bz|<|c|$.
More generally, we have (cf.~\cite[(III.9), (III.10)]{GaspRah})
\begin{align}
\qhyp32{a,b,c}{d,e}{q,\frac{de}{abc}}&=\frac{(\frac ea,\frac{de}{bc};q)_\infty}{(e,\frac{de}{abc};q)_\infty}
\qhyp32{a,\frac db,\frac dc}{d,\frac{de}{bc}}{q,\frac{e}{a}}\label{eq:32-1}\\
&=\frac{(b,\frac{de}{ab},\frac{de}{bc};q)_\infty}{(d,e,\frac{de}{abc};q)_\infty}
\qhyp32{\frac db,\frac eb,\frac{de}{abc}}{\frac{de}{ab},\frac{de}{bc}}{q,b},\label{eq:32-2}
\end{align}
valid for $|de|<|abc|$, $|e|<|a|$, and $|b|<1$.
In this paper we will make use of the following three-term $_2\phi_1$ transformation~\cite[(III.32)]{GaspRah}
\begin{align}
\qhyp21{a,b}{c}{q,z}&=\frac{(b,\frac{c}{a},az,\frac{q}{az};q)_\infty}{(
c,\frac{b}{a},z,\frac qz;q)_\infty}\qhyp21{a,\frac{qa}{c}}{\frac{qa}{b}}{q,\frac{qc}{abz}}
+\frac{(a,\frac{c}{b},bz,\frac q{bz};q)_\infty}{(c,\frac{a}{b},z,\frac qz;q)_\infty}
\qhyp21{b,\frac{qb}{c}}{\frac{qb}{a}}{q,\frac{qc}{abz}},
\label{2phi1-3term}
\end{align}
valid for $\max(|z|,|\frac{qc}{abz}|)<1$.
There is also a useful transformation of a ${}_2\phi_1$ expressed as a sum of two ${}_3\phi_2$'s with one vanishing denominator element \cite[\href{http://dlmf.nist.gov/17.9.E3_5}{(17.9.3\_5)}]{NIST:DLMF}
\begin{equation}
\qhyp21{a,b}{c}{q,z}=\frac{(\frac{c}{a},\frac{c}{b};q)_\infty}{(
c,\frac{c}{ab};q)_\infty}\qphyp311{a,b,\frac{abz}{c}}{\frac{qab}{c}}{q,q}+
\frac{(a,b,\frac{abz}{c};q)_\infty}{(c,z,\frac{ab}{c};q)_\infty}\qphyp311{\frac{c}{a},\frac{c}{b},z}{\frac{qc}{ab}}{q,q},
\label{2phi3phi2dzero}
\end{equation}
valid for $|z|<1$.
We will also rely upon the following three-term nonterminating transformation for a special ${}_3\phi_2$
\cite[\href{http://dlmf.nist.gov/17.9.E13}{(17.9.13)}]{NIST:DLMF}
\begin{equation}
\qhyp32{a,b,c}{d,e}{q,\frac{de}{abc}}=\frac{(\frac{e}{b},\frac{e}{c};q)_\infty}{(e,\frac{e}{bc};q)_\infty}\qhyp32{\frac{d}{a},b,c}{d,\frac{qbc}{e}}{q,q}+\frac{(\frac{d}{a},b,c,\frac{de}{bc};q)_\infty}
{(d,e,\frac{bc}{e},\frac{de}{abc};q)_\infty}
\qhyp32{\frac{e}{b},\frac{e}{c},\frac{de}{abc}}{\frac{de}{bc},\frac{qe}{bc}}{q,q},
\label{3phi2nbtran}
\end{equation}
where $|de|<|abc|$.
One also has the important transformation due to Bailey 
which gives a nonterminating ${}_8W_7$ in terms of a sum of two nonterminating balanced ${}_4\phi_3$'s 
\cite[\href{http://dlmf.nist.gov/17.9.E16}{(17.9.16)}]{NIST:DLMF}
\begin{eqnarray}
&&\hspace{-0.9cm}{}_8W_7\left(a;b,c,d,e,f;q,\frac{q^2a^2}{bcdef}\right)
=\frac{(qa,\frac{qa}{de},\frac{qa}{df},\frac{qa}{ef};q)_\infty}
{(\frac{qa}{def},\frac{qa}{d},\frac{qa}{e},\frac{qa}{f};q)_\infty}
\qhyp43{\frac{qa}{bc},d,e,f}{\frac{qa}{b},\frac{qa}{c},\frac{def}{a}}
{q,q}\nonumber\\*&&
\hspace{3cm}+\frac{(qa,\frac{q^2a^2}{bdef},\frac{q^2a^2}{cdef},\frac{qa}{bc},d,e,f
;q)_\infty}
{(\frac{q^2a^2}{bcdef},\frac{def}{qa},\frac{qa}{b},\frac{qa}{c},\frac{qa}{d},\frac{qa}{e},\frac{qa}{f};q)_\infty}
\qhyp43{\frac{q^2a^2}{bcdef},\frac{qa}{de},\frac{qa}{df},\frac{qa}{ef}}
{\frac{q^2a^2}{bdef},\frac{q^2a^2}{cdef},\frac{q^2a}{def}}{q,q},
\label{Baileytran}
\end{eqnarray}
where $|q^2a^2|<|bcdef|$, with no vanishing denominator elements.

By taking advantage of Bailey's transformation of a very-well-poised
nonterminating ${}_8W_7$ \eqref{Baileytran}, we can obtain some interesting 
single basic hypergeometric series transformation formulas. 
{The following sequence of
transformation formulas is implied by \cite[Figure 2]{vandeBultRains09}.}


\begin{thm}
\label{thm21}
Let $0<|q|<1$, $z,a,b,c,d,e\in\CCast$. Then
\begin{eqnarray}
\label{W761tran}
&&\hspace{-0.6cm}\qpWhyp{7}{6}{1}{a}{b,c,d,e}{q,\frac{q^2a^2}{bcde}}\!=\!
\frac{(qa,\frac{qa}{de};q)_\infty}{(\frac{qa}{d},\frac{qa}{e};q)_\infty}\qhyp32{\frac{qa}{bc},d,e}{\frac{qa}{b},\frac{qa}{c}}{q,\frac{qa}{de}},
\\
\label{W652tran}
&&\hspace{-0.6cm}\qpWhyp{6}{5}{2}{a}{b,c,d}{q,\frac{q^2a^2}{bcd}}\!=\!
\frac{(qa;q)_\infty}{(\frac{qa}{d};q)_\infty}\qhyp22{\frac{qa}{bc},d}{\frac{qa}{b},\frac{qa}{c}}{q,\frac{qa}{d}}\!=\!
\frac{(qa,\frac{qa}{bd};q)_\infty}{(\frac{qa}{b},\frac{qa}{d};q)_\infty}\qhyp21{b,d}{\frac{qa}{c}}{q,\frac{qa}{bd}},\\
\label{W543tran}
&&\hspace{-0.6cm}\qpWhyp{5}{4}{3}{a}{b,c}{q,\frac{q^2a^2}{bc}}\!=\!
(qa;q)_\infty\qhyp12{\frac{qa}{bc}}{\frac{qa}{b},\frac{qa}{c}}{q,qa}
\!=\!
\frac{(qa;q)_\infty}{(\frac{qa}{b};q)_\infty}\qhyp11{b}{\frac{qa}{c}}{q,\frac{qa}{b}}\nonumber\\*
&&\hspace{3.15cm}\!=\!\frac{(qa,\frac{qa}{bc};q)_\infty}{(\frac{qa}{b},\frac{qa}{c};q)_\infty}\qphyp201{b,c}{-}{q,\frac{qa}{bc}}\!=\!\frac{(qa,b;q)_\infty}{(\frac{qa}{c};q)_\infty}
\qphyp11{-1}{\frac{qa}{bc}}{\frac{qa}{b}}{q,b},
\\
\label{W434tran}
&&\hspace{-0.6cm}\qpWhyp{4}{3}{4}{a}{b}{q,\frac{q^2a^2}{b}}\!=\!
(qa;q)_\infty\qhyp01{-}{\frac{qa}{b}}{q,qa}
\!=\!
\frac{(qa;q)_\infty}{(\frac{qa}{b};q)_\infty}\qphyp101{b}{-}{q,\frac{qa}{b}}\nonumber\\*
&&\hspace{2.85cm}\!=\!(qa,b;q)_\infty
\qphyp01{-2}{-}{\frac{qa}{b}}{q,b},
\\
\label{W325tran}
&&\hspace{-0.6cm}\qpWhyp{3}{2}{5}{a}{-}{q,q^2a^2}\!=\!
(qa;q)_\infty\qphyp001{-}{-}{q,qa}
\!=\!(qa,\pm iq\sqrt{a};q)_\infty\!\qphyp03{-2}{-}{-q,\pm iq\sqrt{a}}{q,qa},
\end{eqnarray}
whenever the series converge. 
\end{thm}

\begin{proof}
For \eqref{W761tran}, start with Bailey's transformation of a nonterminating very-well-poised ${}_8W_7$ to a sum of two nonterminating balanced ${}_4\phi_3$'s
\eqref{Baileytran}
and take the limit $b\to\infty$ (or $c\to\infty$), replace $f\mapsto b$ (or $f\mapsto c$),
and then apply the three-term nonterminating transformation for a ${}_3\phi_2$
\eqref{3phi2nbtran}.
For \eqref{W652tran}, start with \eqref{W761tran},
take the limit as $e\to\infty$ (or $d\to\infty$) 
produces the ${}_2\phi_2$ representation and taking 
the limit as $c\to\infty$ (or $b\to\infty$) and then 
replacing variables accordingly so that $b,c,d$ remains, produces
the result.
For \eqref{W543tran}, starting with \eqref{W652tran}:~(i) 
taking the limit in the ${}_2\phi_2$ representation as 
$d\to\infty$ produces the ${}_1\phi_2$ representation; 
(ii) taking the limit as $b\to\infty$ (or $c\to\infty$) 
in the ${}_2\phi_2$ representation or taking the limit 
as $b\to\infty$ (or $d\to\infty$) in the ${}_2\phi_1$ 
representation produces the ${}_1\phi_1$ representation; 
(iii) taking the limit as $c\to\infty$ in the ${}_2\phi_1$ 
representation produces the ${}_2\phi_0^1$ representation.
For \eqref{W434tran}, starting with \eqref{W543tran}:~(i) 
taking the limit as $c\to\infty$ (or $b\to\infty$) in 
the ${}_1\phi_2$ representation, or taking the limit 
$b\to\infty$ in the ${}_1\phi_1$ representation 
produces the ${}_0\phi_1$ representation;
(ii) taking the limit $c\to\infty$ in the ${}_1\phi_1$ 
representation produce, or taking the limit as $
c\to\infty$ (or $b\to\infty$) in the ${}_2\phi_0^1$ 
representation 
produces 
the ${}_1\phi_0^1$ representation.
For \eqref{W325tran}, starting with \eqref{W434tran}, 
taking the limit as $b\to\infty$ in either the 
${}_0\phi_1$ or ${}_1\phi_0^1$ representations produces the first result. The second result in \eqref{W325tran} is obtained by using \cite[(7.6)]{GesselStanton83}.
This
completes the proof.
\end{proof}

\begin{rem} 
The transformations of the special ${}_3\phi_2(a,b,c;d,e;q,de/(abc))$ which appears in \eqref{W761tran}
is alluded 
to in \cite[p.~20]{vandeBultRains09}. 
Also for the nonterminating transformation of the general 
${}_2\phi_1(a,b;c;q,z)$ in \eqref{W652tran}, see \cite[p.~23]{vandeBultRains09}.
See especially \cite[Figure 2]{vandeBultRains09} 
for a summary of the entire hierarchy, which 
follows from their analysis of the $m=1$ limits 
of functions associated with vertices of various polytopes.
\end{rem}

\begin{rem}
The function appearing in \eqref{W325tran} 
\[
A_q(z):=\sum_{n=0}^\infty \frac{q^{n^2}(-z)^n}{(q;q)_n}=\qphyp001{-}{-}{q,-qz},
\] 
\cite[p.~57]{RamanujanLostNotebook} is what Ismail and collaborators refer to as the {\it Ramanujan function}. This function appears in the Plancherel--Rotach asymptotics for the Al-Salam--Chihara polynomials, see Dai, Ismail \& Wang (2020) \cite{DaiIsmailWang2020}.
The second representation of the Ramanujan function is given in Gessel \& Stanton (1983)
\cite[(7.6)]{GesselStanton83}.
\end{rem}

\section{Bilateral basic hypergeometric series}
\label{sec:bil}

For (ordinary) bilateral hypergeometric series, denoted as $_rH_s$ series,
we refer the reader to the classic textbook by Slater \cite{Slater66}.
A bilateral basic hypergeometric series (cf.~\cite[Ch.~5]{GaspRah}) is defined as 
\begin{equation}
\qpsi{r}{s}{a_1,...,a_r}
{b_1,...,b_s}
{q,z}
:=\sum_{k=-\infty}^\infty
\frac{(a_1,...,a_r;q)_k}
{(b_1,...,b_s;q)_k}
\left((-1)^kq^{\binom k2}\right)^{s-r}
z^k,
\label{bbhs}
\end{equation}
where $b_1,\ldots,b_s\notin q^{-\mathbb N_0}$.
This series converges for 
$s\ge r$ provided that $|b_1\cdots b_s|<|a_1\cdots a_rz|$, and also in the case when $s=r$, $|z|<1$. If $s>r$ and there are no vanishing numerator elements, then the bilateral series represents 
a holomorphic function on $z\in{\mathbb C}\setminus\{0\}$ (with essential singularity at $z=0$).
Define the nonterminating {\it very-well-poised} 
bilateral basic hypergeometric series
${}_{r}\Psi_r$ \cite[(2.1.11)]{GaspRah}
\begin{equation}
\label{rPsir}
\Wpsi{r}{r}{a}{a_3,\ldots,a_r}{q,z}
:=\qpsi{r}{r}{\pm q\sqrt{a},a_3,\ldots,a_{r}}
{\pm \sqrt{a},\frac{qa}{a_3},\ldots,\frac{qa}{a_{r}}}{q,z},
\end{equation} 
where $\pm\sqrt{a},\frac{qa}{a_3},\ldots,\frac{qa}{a_{r}}\not\in q^{-\mathbb N_0}$. 

\medskip 
We will use the following notation 
$\nqppsi{r}{s}{m}$, 
$m\ge r-s$,
(an extension of the notation of van de Bult--Rains
\cite[p.~4]{vandeBultRains09}), 
for bilateral basic hypergeometric series when some denominator elements are equal to zero.
Let $p\in\mathbb N_0$. Then define
\begin{equation}\label{bbotzero}
{}_{r}\psi_s^{\,p}\left(\begin{array}{c}{\bf a}\\
{\bf b}\end{array};q,z\right)
:=
\qpsi{r}{s+p}{\bf a}
{{\bf b},
\overbrace{0,\ldots,0}^{p}}{q,z},
\end{equation}
where ${\bf a}:=\{a_1,\ldots,a_r\}$, ${\bf b}:=\{b_1,\ldots,b_s\}$,
$b_1,\ldots,b_s\not\in q^{-\mathbb N_0}\cup\{0\}$,
and
$
\nlqppsi{r}{s}{0}
={}_{r}\psi_{s}
.$
Note that the 
corresponding nonterminating bilateral basic hypergeometric series ${}_r\psi_s$ \eqref{bbhs} with numerator elements equal to zero is divergent, since the condition $|b_1\cdots b_s|<|a_1\cdots a_r z|$ is impossible to satisfy.
The nonterminating bilateral basic hypergeometric series 
$\nlqppsi{r}{s}{m}
({\bf a};{\bf b};q,z)$, ${\bf a}:=\{a_1,\ldots,a_{r}\}$,
${\bf b}:=\{b_1,\ldots,b_s\}$, is well-defined for $s-r+m\ge 0$. 
Note that we will move interchangeably between the
van de Bult--Rains notation and the alternative
notation with vanishing numerator and denominator elements
which are used on the right-hand side of  \eqref{bbotzero}.

Let $p\in\N_0$.
Further, consider the nonterminating very-well-poised 
bilateral basic hypergeometric series with vanishing denominator elements
${}_{r}\Psi_r^p$ \cite[p.~4]{vandeBultRains09}
\begin{equation}
\label{rpPsir}
\qpWpsi{r}{r}{p}{a}{a_3,\ldots,a_r}{q,z}
:=\qppsi{r}{r}{p}{\pm q\sqrt{a},a_3,\ldots,a_{r}}
{\pm \sqrt{a},\frac{qa}{a_3},\ldots,\frac{qa}{a_{r}}}{q,z},
\end{equation} 
where $\pm\sqrt{a},\frac{qa}{a_3},\ldots,\frac{qa}{a_{r}}\not\in q^{-\mathbb N_0}$. Recognize that ${}_{r}\Psi_r^0={}_{r}\Psi_r$. Note that in the case where $p<0$, then the nonterminating bilateral basic very-well-poised hypergeometric series ${}_{r}\Psi_r^p$ is divergent since for the 
corresponding nonterminating bilateral basic hypergeometric series ${}_r\psi_s$ \eqref{bbhs}, the condition $|b_1\cdots b_s|<|a_1\cdots a_r z|$ is impossible to satisfy.
Therefore, we only consider the nonterminating case ${}_{r}\Psi_r^p$ with $p\ge 0$.
We now present some important properties of 
bilateral basic hypergeometric series.

\subsection{Bilateral basic hypergeometric summations}

First there is Ramanujan's ${}_1\psi_1$ summation 
\cite[\href{http://dlmf.nist.gov/17.8.E2}{(17.8.2)}]{NIST:DLMF}
\begin{equation}
\qpsi11{a}{b}{q,z}=\frac{(q,\frac{b}{a},az,\frac{q}{az};q)_\infty}{(b,\frac{q}{a},z,\frac{b}{az};q)_\infty},
\label{Rama1psi1}
\end{equation}
where $\big|\frac ba\big|<|z|<1$. 
We further have Bailey's bilateral ${}_6\psi_6$ summation 
\cite[\href{http://dlmf.nist.gov/17.8.E7}{(17.8.7)}]{NIST:DLMF}
\begin{equation}
\Wpsi66{a}{b,c,d,e}{q,\frac{qa^2}{bcde}}=\qpsi66{\pm q\sqrt{a},b,c,d,e}{\pm\sqrt{a},\frac{qa}{b},\frac{qa}{c},\frac{qa}{d},\frac{qa}{e}}{q,\frac{qa^2}{bcde}}=
\frac{(q,qa,\frac{q}{a},\frac{qa}{bc},\frac{qa}{bd},\frac{qa}{be},\frac{qa}{cd},\frac{qa}{ce},\frac{qa}{de};q)_\infty}{(\frac{q}{b},\frac{q}{c},\frac{q}{d},\frac{q}{e},\frac{qa}{b},\frac{qa}{c},\frac{qa}{d},\frac{qa}{e},\frac{qa^2}{bcde};q)_\infty},
\label{eq:6psi6}
\end{equation}
where $|qa^2|<|bcde|$.

\subsection{Several known bilateral basic hypergeometric transformations}

There are several important transformations of bilateral basic hypergeometric series which we will need.
The first is (cf.~\cite[\href{http://dlmf.nist.gov/17.10.E1}{(17.10.1)}]{NIST:DLMF})
\begin{equation}
\qpsi22{a,b}{c,d}{q,z}=\frac{(\frac{c}{a},\frac{d}{b},az,\frac{qc}{abz};q)_\infty}{(c,\frac{q}{b},z,\frac{cd}{abz};q)_\infty}\qpsi22{a,\frac{abz}{c}}{d,az}{q,\frac{c}{a}},
\label{qpsi22trans}
\end{equation}
valid for $\big|\frac{cd}{ab}\big|<|z|<1$ and $\big|\frac {cd}{ab}\big|<\big|\frac ca\big|<1$.
The next is given in Rosengren \cite[p.~366]{Rosengren2005}
\begin{equation}
\qpsi22{a,b}{c,d}{q,z}=
\frac{(q,b,\frac{c}{a};q)_\infty\vt(az
;q)
}{(\frac{q}{a},c,d,z,\frac{c}{az};q)_\infty}\qhyp21{z,\frac{d}{b}}{\frac{qaz}{c}}{q,\frac{qb}{c}}+
\frac{(q,\frac{q}{c},\frac{d}{b};q)_\infty\vt(\frac{abz}{c}
;q)
}{(d,\frac{q}{a},\frac{q}{b},\frac{az}{c},\frac{cd}{abz};q)_\infty}\qhyp21{\frac{c}{a},\frac{cd}{abz}}{\frac{qc}{az}}{q,\frac{qb}{c}},
\end{equation}
where $|qb|<|c|$, $|cd|<|abz|$, $|z|<1$.
Another important transformation for bilateral basic hypergeometric series is 
Bailey's ${}_6\Psi_6^2$ (sometimes referred as ${}_6\psi_8$) to ${}_2\psi_2$  transformation
cf.~\cite[Exercise 5.11]{GaspRah}
\begin{eqnarray}
&&\hspace{-4.5cm}\qpWpsi662{a}{b,c,d,e}{q,\frac{q^2a^3}{bcde}}
=
\frac{(qa,\frac{q}{a},\frac{qa}{bc},\frac{qa}{de};q)_\infty}{(\frac{q}{b},\frac{q}{c},\frac{qa}{d},\frac{qa}{e};q)_\infty}
\qpsi22{d,e}{\frac{qa}{b},\frac{qa}{c}}{q,\frac{qa}{de}},
\label{Bailey}
\end{eqnarray}
where $|qa|<|bc|$ and $|qa|<|de|$.
This identity was used in \cite{Schlosserbilateral}, in conjunction with the Jacobi triple product identity \eqref{tfdef}, to deduce a number of bilateral identities of the Rogers--Ramanujan type.
In this paper we would like to consider generalizations and repercussions of the important bilateral transformation \eqref{Bailey}. Note that combining \eqref{qpsi22trans} and \eqref{Bailey} gives many transformations of the ${}_6\Psi_6^2$ (or ${}_6\psi_8$).

\subsection{Transformations for the ${}_{10}\Psi_{10}$}

An important transformation for the very-well-poised
${}_{10}\Psi_{10}$ can be obtained by evaluating Slater's \cite{Slater52}   ${}_{2r}\psi_{2r}$  bilateral transformation for $r=5$ (see \cite[(5.5.2)]{GaspRah}). First define the multiset ${\mathbf{b}}:=\{b,c,d,e,f,g,h,k\}$. Then one has
\begin{eqnarray}
&&\hspace{-0.7cm}
\Wpsi{10}{10}{a}{\mathbf{b}}{q,\frac{q^3a^4}{bcdefghk}}:=
\qpsi{10}{10}{\pm q\sqrt{a},\mathbf{b}}{\pm\sqrt{a},\frac{qa}{\mathbf{b}}}{q,\frac{q^3a^4}{bcdefghk}}\nonumber\\[-0.1cm]
&&\hspace{0.3cm}=
\frac{\vt(a;q)}{(1-a)(\frac{q}{\mathbf{b}},\frac{qa}{\mathbf{b}};q)_\infty}\II{
\lambda;\mu,\nu}\frac{(1-\frac{\lambda^2}{a})\vt(\mu,\nu,\frac{\mu}{a},\frac{\nu}{a};q)(\frac{q\lambda}{\mathbf{b}},\frac{qa}{\lambda\mathbf{b}};q)_\infty}{\vt(\frac{\mu}{\lambda},\frac{\nu}{\lambda},\frac{\lambda\mu}{a},\frac{\lambda\nu}{a},\frac{\lambda^2}{a};q)}\Wpsi{10}{10}{\frac{\lambda^2}{a}}{\frac{\lambda \mathbf{b}}{a}}{q,\frac{q^3a^4}{bcdefghk}},
\label{psi101010Psi10}
\end{eqnarray}
provided $|q^3a^4|<|bcdefghk|$, where $\lambda,\mu,\nu$ are free parameters. 
By setting $(\lambda,\mu,\nu)=(g,h,k)$ in \eqref{psi101010Psi10}, one obtains the more-frequently encountered transformation for the very-well-poised ${}_{10}\Psi_{10}$, namely
(cf.~\cite[\href{http://dlmf.nist.gov/17.10.E6}{(17.10.6)}]{NIST:DLMF}, \cite[(III.40)]{GaspRah})
\begin{eqnarray}
&&\hspace{-0.5cm}
\Wpsi{10}{10}{a}{b,c,d,e,f,g,h,k}{q,\frac{q^3a^4}{bcdefghk}}=
\qpsi{10}{10}{\pm q\sqrt{a},b,c,d,e,f,g,h,k}{\pm\sqrt{a},\frac{qa}{b},
\frac{qa}{c},\frac{qa}{d},\frac{qa}{e},\frac{qa}{f},\frac{qa}{g},\frac{qa}{h},\frac{qa}{k}}{q,\frac{q^3a^4}{bcdefghk}}
\nonumber\\[-0.1cm]
&&\hspace{0.0cm}=
\frac{(q,qa,\frac{q}{a};q)_\infty}{(\frac{q}{b},\frac{q}{c},\frac{q}{d},\frac{q}{e},\frac{q}{f},\frac{qa}{b},\frac{qa}{c},\frac{qa}{d},\frac{qa}{e},\frac{qa}{f};q)_\infty}\II{g;h,k}
\frac{(h,k,\frac{h}{a},\frac{k}{a},\frac{qg}{b},\frac{qg}{c},\frac{qg}{d},\frac{qg}{e},\frac{qg}{f},\frac{qa}{bg},\frac{qa}{cg},\frac{qa}{dg},\frac{qa}{eg},\frac{qa}{fg};q)_\infty}{(\frac{q}{g},\frac{h}{g},\frac{k}{g},\frac{gh}{a},\frac{gk}{a},\frac{qa}{g},\frac{qg^2}{a};q)_\infty}\nonumber\\*[0.1cm]
&&\hspace{1.5cm}\times
\Whyp{10}{9}{\frac{g^2}{a}}{\frac{bg}{a},\frac{cg}{a},\frac{dg}{a},\frac{eg}{a},\frac{fg}{a},\frac{gh}{a},\frac{gk}{a}}{q,\frac{q^3a^4}{bcdefghk}},
\label{psi101010W9}
\end{eqnarray}
provided $|q^3a^4|<|bcdefghk|$. 
The above transformation can be seen to be a generalization of \eqref{psi888W7} 
below, obtained by setting $k=qa/h$. 
Slater's very-well-poised ${}_{10}\Psi_{10}$ transformation formula can principally be utilized,
by taking limits of the numerator parameters $b,c,d,e,f$ to infinity, to obtain transformations for the ${}_9\Psi_9^1$, ${}_8\Psi_8^2$, ${}_7\Psi_7^3$, ${}_6\Psi_6^4$, ${}_5\Psi_5^5$ functions. However, these functions are not included in the present study.

\begin{rem}It's important to recognize that even though Slater's very-well-poised transformations \eqref{psi101010Psi10}, \eqref{psi101010W9} are four-term bilateral transformations, they are \textit{not} bilateral extensions of Bailey's four-term very-well-poised balanced ${}_{10}W_9$ transformations \cite[(III.39) and Ex.~2.30]{GaspRah} and also do \textit{not}
contain Bailey's transformation of an $_8W_7$ series in terms of two balanced $_4\phi_3$ series  \cite[(III.36)]{GaspRah}
or in terms of other $_8W_7$ series  \cite[(III.37)]{GaspRah}. \label{remSl}
\end{rem}

\subsection{Transformations for the ${}_8\Psi_8$ series}
In the following we turn to transformations involving ${}_8\Psi_8$ series. In addition to stating Slater's transformation
of ${}_8\Psi_8$ series in \eqref{psi888Psi8}
(which is a special case of \eqref{psi101010Psi10}), our focus is on rather recently
obtained transformations of ${}_8\Psi_8$ series (by Zhang and Zhang~\cite{ZhangZhang2007},
and by Wei and Yu~\cite{WeiYu2021}) that contain the Bailey transformations \cite[(III.36)]{GaspRah} 
and \cite[(III.37)]{GaspRah} mentioned in Remark~\ref{remSl}. These will be reproduced in Theorems~\ref{thm22} and \ref{thmwy},
together with some nice specializations, given as corollaries, that have not been explicitly stated before.

\begin{rem}
We would like to emphasize the difficulty in extending (and possibly even unifying) the results that
are featured in  Theorems~\ref{thm22} and \ref{thmwy} to the ${}_{10}\psi_{10}$ level.
In particular, Zhang and Zhang~\cite{ZhangZhang2007} obtained their result by
making use of Bailey's nonterminating ${}_{10}W_9$ transformation and applying a bilateralization method
(that necessarily
reduces the order of the basic hypergeometric series), while Wei and Yu~\cite{WeiYu2021} used Ismail's
analytic continuation argument from \cite{Ismail77} to lift a three-term transformation of
$_8\phi_7$ series to one that involves two ${}_8\Psi_8$ series and one $_8\phi_7$ series.
While Ismail's argument has proven to be powerful in many cases,
the argument can not be applied in the special setting of \textit{balanced} series
(such as in the series occurring in Bailey's four-term ${}_{10}W_9$ transformation).
While a bilateral series, summed over integers $j$, containing in its summands
the factors $\frac{(\alpha z^{-1};q)_j}{(\beta z;q)_j}z^j$,
can very well be analytic in the variable $z$ (an argument used by Askey and Ismail~\cite{AskeyIsmail79} to extend
Rogers' $_6\phi_5$ summation~\cite[(II.20)]{GaspRah} to Bailey's $_6\psi_6$ summation~\eqref{eq:6psi6}),
a series over $j$ containing the well-poised factors
$$\frac{(\alpha z,\frac{\alpha}z;q)_j}{(\beta z,\frac{\beta}z;q)_j}$$
is \textit{not} analytic in $z$.
Thus, there is unfortunately no clear direct route to obtaining transformations for ${}_{10}\psi_{10}$ series
that would contain the results given in Theorems~\ref{thm22} and \ref{thmwy}.
\end{rem}

An important transformation for the very-well-poised
${}_{8}\Psi_{8}$ can be obtained by evaluating \cite[(5.5.2)]{GaspRah} for $r=4$. Let $\lambda,\mu,a,b,c,d,e,f\in\CCast$ and define the multiset ${\mathbf{b}}:=\{b,c,d,e,f,g\}$. Then one has the following transformation of an ${}_8\Psi_8$ as a symmetric sum of two very-well-poised ${}_8\Psi_8$'s, namely 
\begin{eqnarray}
&&\hspace{-0.7cm}
\Wpsi{8}{8}{a}{\mathbf{b}}{q,\frac{q^2a^3}{bcdefg}}:=
\qpsi{8}{8}{\pm q\sqrt{a},\mathbf{b}}{\pm\sqrt{a},\frac{qa}{\mathbf{b}}}{q,\frac{q^2a^3}{bcdefg}}\nonumber\\[-0.1cm]
&&\hspace{0.3cm}=
\frac{\vt(a;q)}{(1-a)\vt(\frac{\lambda\mu}{a};q)(\frac{q}{\mathbf{b}},\frac{qa}{\mathbf{b}};q)_\infty}\II{
\lambda;\mu}\frac{(1-\frac{\lambda^2}{a})\vt(\mu,\frac{\mu}{a};q)(\frac{q\lambda}{\mathbf{b}},\frac{qa}{\lambda\mathbf{b}};q)_\infty}{\vt(\frac{\mu}{\lambda},\frac{\lambda^2}{a};q)}\Wpsi{8}{8}{\frac{\lambda^2}{a}}{\frac{\lambda \mathbf{b}}{a}}{q,\frac{q^2a^3}{bcdefg}},
\label{psi888Psi8}
\end{eqnarray}
provided $|q^2a^3|<|bcdefg|$, where $\lambda,\mu$ are free parameters.
By setting $(\lambda,\mu)=(f,g)$ in \eqref{psi888Psi8}, one obtains the more-frequently encountered transformation for the very-well-poised ${}_{8}\Psi_{8}$, namely \cite[\href{http://dlmf.nist.gov/17.10.E5}{(17.10.5)}]{NIST:DLMF}, \cite[(III.38)]{GaspRah}, 
a symmetric sum of two very-well-poised ${}_8W_7$'s, namely 
\begin{eqnarray}
&&\hspace{-0.5cm}
\Wpsi88{a}{\mathbf b}{q,\frac{q^2a^3}{bcdefg}}
=\frac{(q,qa,\frac{q}{a};q)_\infty}{(\frac{q}{b},\frac{q}{c},\frac{q}{d},\frac{q}{e},\frac{qa}{b},\frac{qa}{c},\frac{qa}{d},\frac{qa}{e},\frac{fg}{a};q)_\infty}
\nonumber\\*
&&\hspace{1.0cm}\times
\II{f;g}
\frac{(g,\frac{g}{a},\frac{qf}{b},\frac{qf}{c},\frac{qf}{d},\frac{qf}{e},\frac{qa}{bf},\frac{qa}{cf},\frac{qa}{df},\frac{qa}{ef};q)_\infty}{(\frac{q}{f},\frac{g}{f},\frac{qa}{f},\frac{qf^2}{a};q)_\infty}\Whyp87{\frac{f^2}{a}}{\frac{bf}{a},\frac{cf}{a},\frac{df}{a},\frac{ef}{a},\frac{fg}{a}}{q,\frac{q^2a^3}{bcdefg}}.
\label{psi888W7}
\end{eqnarray}
\noindent Note that the transformations \eqref{psi888Psi8}, \eqref{psi888W7} are invariant under
parameter interchange of the six variables $b,c,d,e,f,g$.
For a direct proof of \eqref{psi888W7}, see \cite[\S5.6]{GaspRah}.

\medskip
\noindent As mentioned in Section \ref{nontermbhssec}, we would like to consider generalizations of \eqref{Bailey}. In order to accomplish this, we now examine Zhang and Zhang's transformation \cite[Theorem 6]{ZhangZhang2007} which is 
is a transformation for a very-well-poised ${}_8\Psi_8$ series that constitutes a quite notable
generalization of Bailey's transformation \eqref{Baileytran}. However,
the result of Zhang--Zhang is incorrectly stated in their paper; we provide a correction to it,
and provide it in a form such that the $_8\Psi_8$ series is manifestly symmetric
in six of its parameters, which will be the contents of Theorem~\ref{thm22}.

First we correct the Zhang--Zhang result \cite[Theorem 6]{ZhangZhang2007}. Specifically,
the first balanced ${}_4\phi_3$ in their formula must be replaced by the 
${}_4\phi_3(b,c,d,e;\frac{qa}{f},\frac{qa}{g},\frac{qa}{h};q,q)$ series, which gives
\begin{eqnarray}
&&\hspace{-0.1cm}\qpsi44{a,f,g,h}{\frac{qa}{b},\frac{qa}{c},\frac{qa}{d},\frac{qa}{e}}{q,q}=a\frac{(q,\frac{b}{a},\frac{c}{a},\frac{d}{a},\frac{e}{a},\frac{qa}{f},\frac{qa}{g},\frac{qa}{h};q)_\infty}{(\frac{q}{a},\frac{q}{f},\frac{q}{g},\frac{q}{h},b,c,d,e;q)_\infty}\qhyp43{b,c,d,e}{\frac{qa}{f},\frac{qa}{g},\frac{qa}{h}}{q,q}\nonumber\\*
&&\hspace{3.8cm}+\frac{b}{a}\frac{(q,a,f,g,h,\frac{qb}{c},\frac{qb}{d},\frac{qb}{e};q)_\infty}{(b,\frac{qa}{b},\frac{qa}{c},\frac{qa}{d},\frac{qa}{e},\frac{bf}{a},\frac{bg}{a},\frac{bh}{a};q)_\infty}\qhyp43{b,\frac{bf}{a},\frac{bg}{a},\frac{bh}{a}}{\frac{qb}{c},\frac{qb}{d},\frac{qb}{e}}{q,q}\nonumber\\*
&&\hspace{3.8cm}+\frac{(a,\frac{b}{a},\frac{qa}{\lambda c},\frac{qa}{\lambda d},\frac{qa}{\lambda e},\frac{q\lambda}{f},\frac{q\lambda}{g},\frac{q\lambda}{h};q)_\infty}{(q\lambda,\frac{q}{\lambda},c,d,e,\frac{bf}{a},\frac{bg}{a},\frac{bh}{a};q)_\infty}\qpsi88{\pm q\sqrt{\lambda},\frac{\lambda c}{a},
\frac{\lambda d}{a},\frac{\lambda e}{a},f,g,h}
{\pm\sqrt{\lambda},\frac{qa}{c},\frac{qa}{d},\frac{qa}{e},\frac{q\lambda}{f},\frac{q\lambda}{g},\frac{q\lambda}{h}}{q,b},
\label{eq:ZZc}
\end{eqnarray}
where $\lambda=\frac{qa^2}{cde}$, $h=\frac{q^2a^3}{bcdefg}$.
Now, with this correction, writing out $\lambda$ and $h$ by their defining relations, and
making the replacement
\[
\{a,b,c,d,e,f,g\}\mapsto\left\{\frac{qa^2}{bcd},\frac{q^2a^3}{bcdefg},\frac{qa}{cd},\frac{qa}{bd},\frac{qa}{bc},e,f\right\},
\]
the corrected result can be restated (now with the very-well-poised ${}_8\psi_8$ series on one side alone) as follows.

\begin{thm}{Zhang and Zhang (2007) \cite[Theorem 6]{ZhangZhang2007}.}
\label{thm22}
Let $0<|q|<1$, $a,b,c,d,e,f,g\in\CCast$ such that
$|q^2a^3|<|bcdefg|$ such that none of the denominator factors vanish. Then one has the following transformation formula for a very-well-poised ${}_8\Psi_8$, namely 
\begin{eqnarray}
&&\hspace{-1.0cm}\Wpsi88{a}{b,c,d,e,f,g}{q,\frac{q^2a^3}{bcdefg}}=\qpsi88{\pm q\sqrt{a},b,c,d,e,f,g}{\pm\sqrt{a},\frac{qa}{b},
\frac{qa}{c},\frac{qa}{d},\frac{qa}{e},\frac{qa}{f},\frac{qa}{g}}{q,\frac{q^2a^3}{bcdefg}}\nonumber\\
&&\hspace{0.0cm}=
\frac{(qa,\frac{q}{a},\frac{qa}{bc},\frac{qa}{bd},\frac{qa}{cd},\frac{qa}{ef},\frac{qa}{eg},\frac{qa}{fg};q)_\infty}
{(\frac{q}{b},\frac{q}{c},\frac{q}{d},\frac{qa}{e},\frac{qa}{f},\frac{qa}{g},\frac{qa^2}{bcd},\frac{qa}{efg};q)_\infty}
\qpsi44{e,f,g,\frac{qa^2}{bcd}}{\frac{qa}{b},\frac{qa}{c},\frac{qa}{d},\frac{efg}{a}}{q,q}\nonumber\\*
&&\hspace{0.5cm}+\frac{(q,qa,\frac{q}{a},\frac{b}{a},\frac{c}{a},\frac{d}{a},\frac{qa}{ef},\frac{qa}{eg},\frac{qa}{fg},\frac{q^2a^2}{bcde},\frac{q^2a^2}{bcdf},\frac{q^2a^2}{bcdg};q)_\infty}
{(\frac{q}{b},\frac{q}{c},\frac{q}{d},\frac{q}{e},\frac{q}{f},\frac{q}{g},\frac{qa}{e},\frac{qa}{f},\frac{qa}{g},\frac{bcd}{qa^2},\frac{q^2a^2}{bcd},\frac{q^2a^3}{bcdefg};q)_\infty}
\qhyp43{\frac{qa}{bc},\frac{qa}{bd},\frac{qa}{cd},\frac{q^2a^3}{bcdefg}}{\frac{q^2a^2}{bcde},\frac{q^2a^2}{bcdf},\frac{q^2a^2}{bcdg}}{q,q}\nonumber\\*
&&\hspace{0.5cm}+\frac{(q,qa,\frac{q}{a},e,f,g,\frac{qa}{bc},\frac{qa}{bd},\frac{qa}{cd},\frac{q^2a^2}{befg},\frac{q^2a^2}{cefg},\frac{q^2a^2}{defg};q)_\infty}{(\frac{q}{b},\frac{q}{c},\frac{q}{d},\frac{qa}{b},\frac{qa}{c},\frac{qa}{d},\frac{qa}{e},\frac{qa}{f},\frac{qa}{g},\frac{efg}{qa},\frac{q^2a}{efg},\frac{q^2a^3}{bcdefg};q)_\infty}\qhyp43{\frac{qa}{ef},
\frac{qa}{eg},\frac{qa}{fg},\frac{q^2a^3}{bcdefg}}{\frac{q^2a^2}{befg},\frac{q^2a^2}{cefg},\frac{q^2a^2}{defg}}{q,q}.
\label{8psi84psi4}
\end{eqnarray}
\end{thm}
\begin{rem}
Letting $(g,d)\mapsto(q^{-n},q^{-n}a)$ in Theorem~\ref{thm22} followed by replacing $f\mapsto d$, reproduces \cite[\href{http://dlmf.nist.gov/17.10.E3}{(17.10.3)}]{NIST:DLMF},
which, as a transformation involving terminating series on both sides,
is equivalent to Watson's transformation~\cite[(III.18)]{GaspRah}.
Instead, by only letting $g\mapsto q^{-n}$ (or $d\mapsto q^{-n}a$) in \eqref{eq:ZZc},
an identity equivalent to Bailey's transformation in \eqref{Baileytran} is obtained.
\end{rem}

If, in \eqref{eq:ZZc}, one substitutes  $(a,b,c,d,e,f,g)\mapsto(q^a,q^b,q^c,q^d,q^e,q^f,q^g)$
and uses the definition of the $q$-gamma function,
then one obtains the following transformation for very-well-poised bilateral ${}_7H_7$ to a ${}_4H_4$ and two nonterminating balanced ${}_4F_3$'s, all series with argument unity.
\begin{thm}
\label{thm337H7}
Let $a,b,c,d,e,f,g\in\CCast$. Then
\begin{eqnarray}
&&\Hhyp77{1\!+\!\frac{a}{2},b,c,d,e,f,g}{\frac{a}{2},1\!+\!a\!-\!b,1\!+\!a\!-\!c,1\!+\!a\!-\!d,1\!+\!a\!-\!e,1\!+\!a\!-\!f,1\!+\!a\!-\!g}{1}
\nonumber\\
&&\hspace{0.0cm}=
\frac{\Gamma(1\!-\!b,1\!-\!c,1\!-\!d,1\!+\!a\!-\!e,1\!+\!a\!-\!f,1\!+\!a\!-\!g)}{\Gamma(1\!+\!a,1\!-\!a)}\Biggl(\frac{\Gamma(1\!+\!2a\!-\!b\!-\!c\!-\!d,1\!+\!a\!-\!e\!-\!f\!-\!g)}{\Gamma(1\!+\!a\!-\!b\!-\!c,1\!+\!a\!-\!b\!-\!d,1\!+\!a\!-\!c\!-\!d,1\!+\!a\!-\!e\!-\!f)}\nonumber\\
&&\hspace{4.15cm}\times\frac{1}{\Gamma(1\!+\!a\!-\!e\!-\!g,1\!+\!a\!-\!f\!-\!g)}\Hhyp44{e,f,g,1\!+\!2a\!-\!b\!-\!c\!-\!d}{1\!+\!a\!-\!b,1\!+\!a\!-\!c,1\!+\!a\!-\!d,-a\!+\!e\!+\!f\!\!+g}{1}\nonumber\\*
&&\hspace{0.0cm}+\frac{\Gamma(1\!-\!e,1\!-\!f,1\!-\!g,\!-\!1\!-\!2a\!+\!b\!+\!c\!+\!d,2\!+\!2a\!-\!b\!-\!c\!-\!d)}{\Gamma(b\!-\!a,c\!-\!a,d\!-\!a,1\!+\!a\!-\!e\!-\!f,1\!+\!a\!-\!e\!-\!g,1\!+\!a\!-\!f\!-\!g,2\!+\!2a\!-\!b\!-\!c\!-\!d\!-\!e,2\!+\!2a\!-\!b\!-\!c\!-\!d\!-\!f)}\nonumber\\
&&\hspace{0.5cm}\times\frac{\Gamma(2\!+\!3a\!-\!b\!-\!c\!-\!d\!-\!e\!-\!f\!-\!g)}{\Gamma(2\!+\!2a\!-\!b\!-\!c\!-\!d\!-\!g)}
\hyp43{1\!+\!a\!-\!b\!-\!c,1\!+\!a\!-\!b\!-\!d,1\!+\!a\!-\!c\!-\!d,2\!+\!3a\!-\!b\!-\!c\!-\!d\!-\!e\!-\!f\!-\!g}{2\!+\!2a\!-\!b\!-\!c\!-\!d\!-\!e,2\!+\!2a\!-\!b\!-\!c\!-\!d\!-\!f,2\!+\!2a\!-\!b\!-\!c\!-\!d\!-\!g}{1}\nonumber\\*
&&\hspace{0.0cm}
+\frac{\Gamma(1\!+\!a\!-\!b,1\!+\!a\!-\!c,1\!+\!a\!-\!d,-1\!-\!a\!+\!e\!+\!f\!+\!g,2\!+\!a\!-\!e\!-\!f\!-\!g,2\!+\!3a\!-\!b\!-\!c\!-\!d\!-\!e\!-\!f\!-\!g)}{\Gamma(e,f,g,1\!+\!a\!-\!b\!-\!c,1\!+\!a\!-\!b\!-\!d,1\!+\!a\!-\!c\!-\!d,2\!+\!2a\!-\!b\!-\!e\!-\!f\!-\!g,2\!+\!2a\!-\!c\!-\!e\!-\!f\!-\!g,2\!+\!2a\!-\!d\!-\!e\!-\!f\!-\!g)}\nonumber\\*
&&\hspace{4.85cm}\times
\hyp43{1\!+\!a\-\!-\!e\!-\!f,1\!+\!a\!-\!e\!-\!g,1\!+\!a\!-\!f\!-\!g,2\!+\!3a\!-\!b\!-\!c\!-\!d\!-\!e\!-\!f\!-\!g}{2\!+\!2a\!-\!b\!-\!e\!-\!f\!-\!g,2\!+\!2a\!-\!c\!-\!e\!-\!f\!-\!g,2\!+\!2a\!-\!d\!-\!e\!-\!f\!-\!g}{1}
\Biggr),\nonumber\\*
\end{eqnarray}
provided $\Re(2+3a-b-c-d-e-f-g)>0$.
\end{thm}
\begin{proof}
Starting with the Zhang--Zhang transformation in Theorem~\ref{thm22}, we substitute
\[
(a,b,c,d,e,f,g)\mapsto(q^a,q^b,q^c,q^d,q^e,q^f,q^g)
\]
and rewrite ratios of infinite $q$-shifted factorials using the $q$-gamma function, defined by
\cite[\href{http://dlmf.nist.gov/5.18.E4}{(5.18.4)}]{NIST:DLMF}
\begin{equation}
\Gamma_q(z):=\frac{(q;q)_\infty(1-q)^{1-z}}{(q^z;q)_\infty}.
\end{equation}
Then we take the limit as $q\to1^{-}$ on both sides\footnote{While it is a routine matter to
\textit{formally} take the limit $q\to1^{-}$ (and is useful to \textit{find} the desired result), a rigorous
justification requires an explanation why taking the term-wise limit in the summands of the series is allowed.
This typically requires that the series is majorized by a convergent series not depending on $q$ and follows
the arguments provided by Koornwinder in \cite[Appendix~A]{Koornwinder1990}.}, using 
$\lim_{q\to1^{-}}\Gamma_q(z)=\Gamma(z)$
and applying  \cite[(1.8.4)]{Koekoeketal}
\begin{equation}
\lim_{q\to1^{-}}\frac{(q^\alpha;q)_k}{(1-q)^k}=(\alpha)_k,
\end{equation}
where $(\alpha)_k:=(\alpha)(\alpha+1)\cdots(\alpha+k-1)$ is a rising factorial.
The expression is thus reduced from the respective bilateral and unilateral
basic hypergeometric functions to corresponding ordinary hypergeometric functions, completing the proof
(where we have omitted the technical details pertaining to rigourously justifying the term-wise $q\to1^{-}$ limit).
\end{proof}

\begin{rem}
Note that one can obtain similar transformations of bilateral hypergeometric series by taking the limit as $q\to1^{-}$ of transformations which are derived below. 
However, the details are routine and are omitted.
\end{rem}

\begin{rem}As noticed in \cite[Corollary 8]{ZhangZhang2007}, setting $c=a$ in Theorem~\ref{thm22} and replacing $g\mapsto c$ produces Bailey's transformation of a nonterminating ${}_8W_7$ in terms of a sum of two balanced nonterminating ${}_4\phi_3$'s with argument $q$ \eqref{Baileytran}.
\end{rem}

\noindent The transformation in Theorem~\ref{thm22} involves series that are manifestly symmetric in $b,c,d$ and,
separately, in $e,f,g$.
Using the simple identity $(z,\omega z,\omega^2 z;q)_k=(z^3;q^3)_k$, where $\omega$ is a fixed primitive third root of unity,
we are able to derive the following cubic summation formula involving a bilateral series:

\begin{cor}\label{corbcubic} Let $0<|q|<1$ and $a,b\in\CCast$. Then
\begin{align}
&\Wpsi88{a}{b^{\frac 13},\omega b^{\frac 13},\omega^2 b^{\frac 13},
q^{\frac 13}ab^{-\frac 13},\omega q^{\frac 13}ab^{-\frac 13},
\omega^2 q^{\frac 13}ab^{-\frac 13}}{q,q}
=\sum_{k=-\infty}^\infty\frac{(1-aq^{2k})}{(1-a)}\frac{(b,\frac{qa^3}b;q^3)_k}{(q^2b,\frac{q^3a^3}b;q^3)_k}q^k\notag\\
&=\frac{(qa,\frac qa;q)_\infty}{(\frac{q^3}b,q^2b;q^3)_\infty\vt(\frac b{a^2};q)}
\Bigg(\frac{(q^2,q^3,\frac{qb^2}{a^3};q^3)_\infty\,\big(\vt(\frac{b^2}{a^3},\frac{qb}{a^3};q^3)+
\frac ba\,\vt(\frac b{a^3},\frac{qa^3}{b^2};q^3)\big)}{(q,q,\frac{q^2b}{a^3},\frac{q^3a^3}b,\frac{b^2}{a^3};q^3)_\infty}\notag\\
&\qquad\qquad\qquad\qquad\qquad\quad+\frac{\big(\frac{a^2}b\,\vt(\frac{b}{a^3},\frac{qb^2}{a^3};q^3)-
\frac{b}{a^2}\,\vt(\frac{qa^3}b,\frac{a^3}{b^2};q^3)\big)}{(1-\frac{a^3}{b^2})\,(\frac{q^2b}{a^3},\frac{q^3a^3}b;q^3)_\infty}
\qhyp21{\frac{qb^2}{a^3},q^3}{\frac{q^3b^2}{a^3}}{q^3,q}\Bigg),\label{eq:corbcubic}
\end{align}
where $\omega$ is a fixed primitive third root of unity.
\end{cor}
\begin{proof}
We sketch the details.
In \eqref{8psi84psi4}, we make the substitutions
\begin{equation*}
(b,c,d,e,f,g)\mapsto(b^{\frac 13},\omega b^{\frac 13},\omega^2 b^{\frac 13},
q^{\frac 13}ab^{-\frac 13},\omega q^{\frac 13}ab^{-\frac 13},
\omega^2 q^{\frac 13}ab^{-\frac 13}),
\end{equation*}
where $\omega$ is a fixed primitive third root of unity,
and where for $q^{\frac 13}$ and $b^{\frac 13}$
fixed branches of the respective third roots are taken.
Under these substitutions, the argument of the $_8\Psi_8$ series becomes $q$,
the $_4\Psi_4$ reduces to a $_1\Psi_1$ series of base $q^3$,
and the two $_4\phi_3$ series reduce to Fine--type basic hypergeometric series with base $q^3$
(i.e., to series of the form $_2\phi_1(a,q^3;c;q^3,z)$ which can also be viewed
as the half of a bilateral $_1\Psi_1$ series).
The $_1\Psi_1$ series can be simplified by virtue of the $(q,a,b,z)\mapsto(q^3,\frac{qa^3}b,\frac{q^3a^3}b,q)$
case of Ramanujan's $_1\Psi_1$ summation formula \eqref{Rama1psi1}, while to the first
$_2\phi_1$ series we apply the $(q,a,b,c,z)\mapsto(q^3,\frac{q^3a^3}{b^2},q^3,\frac{q^5a^3}{b^2},q)$
case of \eqref{2phi1-3term}. Simplification yields the claimed result.
\end{proof}

\noindent By letting $b=a^3$ in Corollary~\ref{corbcubic} we immediately obtain the following simpler identity
involving unilateral series:
\begin{cor}\label{corcubic} Let $0<|q|<1$ and $a\in\CCast$. Then
\begin{align}
&\Whyp87{a}{\omega a,\omega^2 a,
q^{\frac 13},\omega q^{\frac 13},
\omega^2 q^{\frac 13}}{q,q}=
\sum_{k=0}^\infty\frac{(1-aq^{2k})}{(1-a)}\frac{(q,a^3;q^3)_k}{(q^3,q^2a^3;q^3)_k}q^k\notag\\
&=\frac{1}{(1-a)\,(q^2a^3;q^3)_\infty}\Bigg(\frac{(q^2,qa^3;q^3)_\infty}{(q;q^3)_\infty}
-\frac{a\,(q,q^3a^3;q^3)_\infty}{(q^3;q^3)_\infty}
\qhyp21{qa^3,q^3}{q^3a^3}{q^3,q}\Bigg),\label{eq:corcubic}
\end{align}
where $\omega$ is a fixed primitive third root of unity.
\end{cor}
\begin{rem}
One can obtain a closed summation from Corollary~\ref{corbcubic} by specializing the parameters
such that the ${}_2\phi_1$ reduces to its first term. This is achieved by letting $qb^2=a^3$. In this case,
several of the terms in \eqref{eq:corbcubic} vanish and on the right-hand side one is left with only one term
(and in addition, several of the factors cancel).
However, the identity thus obtained is nothing new but just the
$(a,b,c,d,e,q)\mapsto(a,q^{-\frac 12}a^{\frac 32},\omega q^{-\frac 12}a^{\frac 32},
\omega^2 q^{-\frac 12}a^{\frac 32},q^{\frac 12}a^{\frac 12},q)$ special case of Bailey's very-well-poised
$_6\psi_6$ summation \eqref{eq:6psi6}.
\end{rem}
\begin{rem}
Since a cubic specialization of Theorem~\ref{thm22} yielded a new result, one might be tempted to try
a quartic specialization that makes use of the simple identity
$(z, -z, \mathrm iz,-\mathrm iz;q)_k=(z^4;q^4)_k$, where $\mathrm i$ is the imaginary unit.
The purpose of such a specialization may be to reduce the $_4\psi_4$ in \eqref{8psi84psi4} to a summable
$_1\psi_1$ (by Ramanujan's formula in \eqref{Rama1psi1}) with base $q^4$. This works by replacing
$(a,b,c,d,e,f,g,q)$ by $(a,b,\mathrm ib,-\mathrm ib,qa^2b^{-3},\mathrm iqa^2b^{-3},-\mathrm iqa^2b^{-3},q)$
but at the same time reduces the two $_4\phi_3$ series in \eqref{8psi84psi4} to summable $_1\phi_0$ series.
As a result, a summation for a bilateral series is obtained into a sum of three infinite $q$-products.
However, closer inspection of the thus obtained identity reveals that the three
terms of infinite $q$-products can be reduced, by two applications of the theta function addition formula \eqref{tfa},
to a single term, and the bilateral series that is left is just a specialized $_4\psi_4$ series of base $q^4$.
It turns out that the identity obtained in the described manner is actually just a special case of the
well-known bilateral analogue of the $q$-Dixon sum in \cite[(II.32)]{GaspRah}.
\end{rem}
\noindent A principal interesting special case of \eqref{8psi84psi4} corresponds to the case of setting 
$g$
to $a$. In this case the ${}_8\Psi_8$ becomes an ${}_8W_7$ and we produce the following result.
\begin{cor}
\label{cor38}Let $0<|q|<1$, $a,b,c,d,e,f\in\CCast$ such that $|q^2a^2|<|bcdef|$. Then
\begin{eqnarray}
&&\hspace{-1cm}\Whyp87{a}{b,c,d,e,f}{q,\frac{q^2a^2}{bcdef}}=
\frac{(qa,\frac{q}{a},\frac{q}{e},\frac{q}{f},\frac{qa}{bc},\frac{qa}{bd},\frac{qa}{cd},\frac{qa}{ef};q)_\infty}{(q,\frac{q}{b},\frac{q}{c},\frac{q}{d},\frac{qa}{e},\frac{qa}{f},\frac{qa^2}{bcd},\frac{q}{ef};q)_\infty}
\qpsi44{a,e,f,\frac{qa^2}{bcd}}{\frac{qa}{b},\frac{qa}{c},\frac{qa}{d},ef}{q,q}\nonumber\\*
&&\hspace{0.5cm}+\frac{(qa,\frac{b}{a},\frac{c}{a},\frac{d}{a},\frac{qa}{ef},\frac{q^2a}{bcd},\frac{q^2a^2}{bcde},\frac{q^2a^2}{bcdf};q)_\infty}{(\frac{q}{b},\frac{q}{c},\frac{q}{d},\frac{qa}{e},\frac{qa}{f},\frac{bcd}{qa^2},\frac{q^2a^2}{bcd},\frac{q^2a^2}{bcdef};q)_\infty}\qhyp43{\frac{qa}{bc},\frac{qa}{bd},\frac{qa}{cd},\frac{q^2a^2}{bcdef}}{\frac{q^2a}{bcd},\frac{q^2a^2}{bcde},\frac{q^2a^2}{bcdf}}{q,q}\nonumber\\*
&&\hspace{0.5cm}+\frac{(a,qa,\frac{q}{a},e,f,\frac{qa}{bc},\frac{qa}{bd},\frac{qa}{cd},\frac{q^2a}{bef},\frac{q^2a}{cef},\frac{q^2a}{def};q)_\infty}{(\frac{q}{b},\frac{q}{c},\frac{q}{d},\frac{qa}{b},\frac{qa}{c},\frac{qa}{d},\frac{qa}{e},\frac{qa}{f},\frac{ef}{q},\frac{q^2}{ef},\frac{q^2a^2}{bcdef};q)_\infty}\qhyp43{\frac{q}{e},\frac{q}{f},\frac{qa}{ef},\frac{q^2a^2}{bcdef}}{\frac{q^2a}{bef},\frac{q^2a}{cef},\frac{q^2a}{def}}{q,q}.
\end{eqnarray}
\end{cor}
\begin{proof}
Start with \eqref{8psi84psi4} and setting $e=a$ converts the ${}_8\Psi_8$ to an ${}_8W_7$. Then replace $g\mapsto e$ and simplification completes the proof.
\end{proof}

\noindent Next we recall an interesting transformation formula first derived by Wei and Yu \cite[Theorem 3]{WeiYu2021} for a very-well-poised ${}_8\Psi_8$. This transformation formula provides a generalization of the two and three-term transformations for an ${}_8W_7$ \cite[(III.23), (III.37)]{GaspRah}.
Note that in the following transformation, while the left-hand side is manifestly symmetric in the six variables $b,c,d,e,f,g$, the right-hand side satisfies the same symmetry (which otherwise is not obvious).

\begin{thm}{Wei and Yu (2019) \cite[Theorem 3]{WeiYu2021}.}
\label{thmwy}
Let $0<|q|<1$, $a,b,c,d,e,f,g\in\CCast$ such that
$|q^2a^3|<|bcdefg|$, $|qa|<|fg|$, and none of the denominator factors vanish. Then one has the following transformation formula for a very-well-poised ${}_8\Psi_8$, namely 
\begin{eqnarray}
&&\hspace{-0.6cm}\Wpsi88{a}{b,c,d,e,f,g}{q,\frac{q^2a^3}{bcdefg}}=\qpsi88{\pm q\sqrt{a},b,c,d,e,f,g}{\pm\sqrt{a},\frac{qa}{b},
\frac{qa}{c},\frac{qa}{d},\frac{qa}{e},\frac{qa}{f},\frac{qa}{g}}{q,\frac{q^2a^3}{bcdefg}}\nonumber\\
&&\hspace{-0.3cm}=
\frac{(qa,\frac{q}{a},\frac{b}{a},\frac{qa}{cd},\frac{qa}{ce},\frac{qa}{de},\frac{qa}{fg},\frac{bcd}{a},\frac{bce}{a},\frac{bde}{a},\frac{q^2a^2}{bcdef},\frac{q^2a^2}{bcdeg};q)_\infty}
{(\frac{q}{f},\frac{q}{g},\frac{qa}{c},\frac{qa}{d},\frac{qa}{e},\frac{bc}{a},\frac{bd}{a},\frac{be}{a},\frac{qa}{cde},\frac{bcde}{a},\frac{q^2a}{bcde},\frac{q^2a^3}{bcdefg};q)_\infty}
\Wpsi88{\frac{bcde}{qa}}{b,c,d,e,\frac{bcdef}{qa^2},\frac{bcdeg}{qa^2}}{q,\frac{qa}{fg}}\nonumber\\*
&&\hspace{-0.2cm}+\frac{(q,qa,\frac{q}{a},c,d,e,\frac{qb}{c},\frac{qb}{d},\frac{qb}{e},\frac{qb}{f},\frac{qb}{g},\frac{qa}{bf},\frac{qa}{bg};q)_\infty\vt(\frac{bcde}{qa^2};q)}
{(\frac{q}{b},\frac{q}{f},\frac{q}{g},\frac{qa}{b},\frac{qa}{c},\frac{qa}{d},\frac{qa}{e},\frac{qa}{f},\frac{qa}{g},\frac{bc}{a},\frac{bd}{a},\frac{be}{a},\frac{qb^2}{a};q)_\infty\vt(\frac{cde}{qa};q)}
\Whyp87{\frac{b^2}{a}}{\frac{bc}{a},\frac{bd}{a},\frac{be}{a},\frac{bf}{a},\frac{bg}{a}}{q,\frac{q^2a^3}{bcdefg}}
.
\label{8psi88W7}
\end{eqnarray}
\end{thm}


\begin{rem}
\label{rem310a}
Wei and Yu noticed  \cite[p.~3]{WeiYu2021}, if one sets $e=qa/d$ in Theorem~\ref{thmwy}, the left-hand side becomes a ${}_6\Psi_6$, the first term on the right-hand side vanishes, and the ${}_8W_7$ becomes a nonterminating ${}_6W_5$ which can then be summed using \cite[\href{http://dlmf.nist.gov/17.7.E15}{(17.7.15)}]{NIST:DLMF}, and the identity reduces to Bailey's ${}_6\Psi_6$ summation \cite[\href{http://dlmf.nist.gov/17.8.E7}{(17.8.7)}]{NIST:DLMF}.
\end{rem}

\begin{rem}
The bilateral basic hypergeometric transformation \eqref{8psi88W7} is seen to be a generalization of the two-term transformation  \cite[(III.23)]{GaspRah}
by setting $b=a$ and then utilizing \cite[Exercise 2.16]{GaspRah}. It is also a generalization of the three-term transformation \cite[(III.37)]{GaspRah}, obtained by setting one of $c,d,e,f,g$ to $a$.
\end{rem}

\noindent By setting $g=a/b$ in Theorem~\ref{thmwy}, then one obtains the following nonterminating summation formula involving two ${}_8\Psi_8$.
\begin{cor}
\label{cor311a}
Let $0<|q|<1$, $a,b,c,d,e,f\in\CCast$ such that
$|q^2a^2|<|cdef|$, $|qb|<|f|$, and none of the denominator factors vanish. Then one has the following summation formula for an ${}_8\Psi_8$, namely
\begin{eqnarray}
&&\hspace{-0.6cm}\Wpsi88{a}{b,c,d,e,f,\frac{a}{b}}{q,\frac{q^2a^2}{cdef}}\nonumber\\
&&\hspace{-0.3cm}=
\frac{(
qa,\frac{q}{a},\frac{b}{a},\frac{qb}{f},\frac{qa}{cd},\frac{qa}{ce},\frac{qa}{de},\frac{bcd}{a},\frac{bce}{a},
\frac{bde}{a},\frac{q^2a}{cde},\frac{q^2a^2}{bcdef}
;q)_\infty}
{(\frac{q}{f},\frac{qb}{a},\frac{qa}{c},\frac{qa}{d},\frac{qa}{e},\frac{bc}{a},\frac{bd}{a},\frac{be}{a},\frac{qa}{cde},\frac{bcde}{a},\frac{q^2a}{bcde},\frac{q^2a^2}{cdef};q)_\infty
}
\Wpsi88{\frac{bcde}{qa}}{b,c,d,e,\frac{bcdef}{qa^2},\frac{cde}{qa}}{q,\frac{qb}{f}}\nonumber\\
&&\hspace{-0.2cm}+\frac{(q,q,qa,\frac{q}{a},c,d,e,\frac{qb}{c},\frac{qb}{d},\frac{qb}{e},\frac{qb}{f},\frac{qa}{bf};q)_\infty\vt(\frac{bcde}{qa^2};q)}
{(qb,\frac{q}{b},\frac{q}{f},\frac{qb}{a},\frac{qa}{b},\frac{qa}{c},\frac{qa}{d},\frac{qa}{e},\frac{qa}{f},\frac{bc}{a},\frac{bd}{a},\frac{be}{a};q)_\infty\vt(\frac{cde}{qa};q)}.
\label{8psi88W7sum1}
\end{eqnarray}
\end{cor}
\begin{rem}By taking $f=a$ in \eqref{8psi88W7sum1} we immediately obtain
the summation formula
\begin{eqnarray}
&&\hspace{0.5cm}\Whyp87
{a}{b,c,d,e,\frac{a}{b}}{q,\frac{q^2a}{cde}}
=
\frac{(
qa,\frac{b}{a},\frac{qa}{cd},\frac{qa}{ce},\frac{qa}{de},\frac{bcd}{a},\frac{bce}{a},
\frac{bde}{a};q)_\infty}
{(\frac{qa}{c},\frac{qa}{d},\frac{qa}{e},\frac{bc}{a},\frac{bd}{a},\frac{be}{a},\frac{qa}{cde},\frac{bcde}{a};q)_\infty
}
\Whyp87{\frac{bcde}{qa}}{b,c,d,e,\frac{cde}{qa}}{q,\frac{qb}{a}}\nonumber\\
&&\hspace{1.5cm}+\frac{(q,qa,c,d,e,\frac{qb}{c},\frac{qb}{d},\frac{qb}{e};q)_\infty\vt(\frac{bcde}{qa^2};q)}
{(qb,\frac{qa}{b},\frac{qa}{c},\frac{qa}{d},\frac{qa}{e},\frac{bc}{a},\frac{bd}{a},\frac{be}{a};q)_\infty\vt(\frac{cde}{qa};q)},
\label{8W7sumcor}
\end{eqnarray}
valid for $0<|q|<1$, $a,b,c,d,e\in\CCast$ such that
$|q^2a|<|cde|$, $|qb|<|a|$ and none of the denominator factors vanish. 
This summation can also be obtained from specializing \cite[(2.11.1)]{GaspRah}
by taking $bc=a$ (which reduces the third ${}_8W_7$ series to $1$),
application of the transformation of nonterminating $_8W_7$ series in \cite[(2.10.1)]{GaspRah} to the second
${}_8W_7$ series, and relabeling.
\end{rem}

\noindent One can also obtain a nonterminating summation formula involving two ${}_8\Psi_8$ by starting with
Theorem~\ref{thmwy} and utilizing the nonterminating very-well-poised ${}_8W_7$ summation 
\cite[(II.16)]{GaspRah}
\begin{eqnarray}
&&\hspace{-2.0cm}\Whyp87{\sqrt{\frac{ab}{q}}c}{a,b,\pm c,\frac{\sqrt{qab}}{c}}{q,-\frac{\sqrt{q}c}{\sqrt{ab}}}=\frac{(\sqrt{qab}c,
\frac{\sqrt{q}c}{\sqrt{ab}};q)_\infty(qa,qb,\frac{qc^2}{a},\frac{qc^2}{b};q^2)_\infty}{(\sqrt{\frac{qa}{b}}c,\sqrt{\frac{qb}{a}}c;q)_\infty(q,qab,qc^2,\frac{qc^2}{ab};q^2)_\infty}.
\label{GRII.16}
\end{eqnarray}
\begin{cor}
\label{cor311}
Let $0<|q|<1$, $a,b,c,d\in\CCast$ such that
$|qa|<|cd|$, $|q^{1/2}b|<|\sqrt{cd}|$ and none of the denominator factors vanish. Then one has the following summation formula for the sum of two ${}_8\Psi_8$'s, namely
\begin{eqnarray}
&&\hspace{-0.5cm}\Wpsi88{a}{b,c,d,\pm\frac{\sqrt{q}a}{\sqrt{cd}},\frac{cd}{b}}{q,-\frac{qa}{cd}}\nonumber\\
&&\hspace{0.0cm}=
\frac{(
qa,\frac{q}{a},-\frac{q}{b},\frac{b}{a},-\sqrt{\frac{qc}{d}},-\sqrt{\frac{qd}{c}},\frac{qa}{cd},\frac{\sqrt{q}b}{\sqrt{cd}},\frac{bcd}{a},-\sqrt{\frac{qc}{d}}b,-\sqrt{\frac{q2d}{c}}b,-\frac{q^{3/2}a}{(cd)^{3/2}}
;q)_\infty}
{(-\sqrt{qcd},-\sqrt{qcd}b,\frac{qa}{c},\frac{qa}{d},\frac{bc}{a},\frac{bd}{a},\frac{qb}{cd},\frac{\sqrt{qcd}}{a},-\frac{\sqrt{q}}{\sqrt{cd}},-\frac{\sqrt{q}b}{\sqrt{cd}},-\frac{q^{3/2}}{\sqrt{cd}b},-\frac{qa}{cd};q)_\infty
}\nonumber\\
&&\hspace{0.8cm}
\times \Wpsi88{-\frac{\sqrt{cd}b}{\sqrt{q}}
}{\pm b,c,d,-\frac{\sqrt{q}a}{\sqrt{cd}},-\frac{(cd)^{3/2}}{\sqrt{q}a}}{q,\frac{\sqrt{q}b}{\sqrt{cd}}}\nonumber\\
&&\hspace{1.2cm}+\frac{(q,qa,\frac{q}{a},c,d,\frac{qa}{cd},\frac{qa}{cd},-\frac{\sqrt{q}a}{\sqrt{cd}},\frac{\sqrt{qcd}}{b},\frac{qb^2}{cd};q)_\infty\vt(-\sqrt{\frac{cd}{q}}\frac{b}{a};q)(\frac{qbc}{a},\frac{qbd}{a},\frac{q^2ab}{c^2d},\frac{q^2ab}{cd^2};q^2)_\infty}
{(\frac{q}{b},\frac{qa}{b},\frac{qa}{c},\frac{qa}{d},\frac{bc}{a},\frac{bd}{a},\frac{\sqrt{qcd}}{a},\frac{qb}{cd},-\frac{\sqrt{q}b}{\sqrt{cd}},\frac{qab}{cd};q)_\infty\vt(-\sqrt{\frac{cd}{q}};q)(q,qcd,\frac{q^2b^2}{cd},\frac{q^2a^2}{c^2d^2};q^2)_\infty}.
\label{8psi88W7sum}
\end{eqnarray}
\end{cor}
\begin{proof}
Start with Theorem~\ref{thmwy} and replace
\[
(g,f,e)\mapsto\left(-
\frac{\sqrt{qcd}a}{be}
,-e,-
\frac{\sqrt{q}a}{\sqrt{cd}}\right).
\]
Then the resulting very-well-poised ${}_8W_7$ can be summed using \cite[(II.16)]{GaspRah}. Simplifying the final result completes the proof. 
\end{proof}

\noindent One can also obtain a nonterminating summation formula involving two ${}_8\Psi_8$ by starting with Theorem~\ref{thmwy} 
and utilizing the nonterminating very-well-poised ${}_8W_7$ summation 
\cite[(II.18)]{GaspRah}
\begin{eqnarray}
&&\hspace{-4.7cm}\Whyp87{a}
{-a,b,\frac{q}{b},c,
\frac{q}{c}}{q,-a}
=\frac{(a,qa;q)_\infty(abc,\frac{qab}{c},\frac{qac}{b},\frac{q^2a}{bc};q^2)_\infty}{(ab,ac,\frac{qa}{b},\frac{qa}{c};q)_\infty}.
\label{GRII.18}
\end{eqnarray}
\begin{cor}
\label{cor312}
Let $0<|q|<1$, $a,b,c,d\in\CCast$ such that
$|b^2|<|a|$, and none of the denominator factors vanish. Then one has the following summation formula for the sum of two ${}_8\Psi_8$'s related by the substitution $(a,d)\mapsto(-a,-d)$, namely
\begin{eqnarray}
&&\hspace{-0.6cm}\Wpsi88{a}{\pm b,c,d,\frac{qa^2}{b^2c},\frac{qa^2}{b^2d}}{q,-\frac{b^2}{a}}\nonumber\\
&&\hspace{-0.3cm}=
\frac{(
qa,\frac{q}{a},\frac{b}{a},-\frac{q}{d},\frac{b^2}{a},\frac{qa}{b},-\frac{qa}{c},-\frac{bc}{a},-\frac{b^2c}{a},-\frac{qa}{bc},
-\frac{b^2d}{a^2}
;q)_\infty}
{(-qa,-\frac{q}{a},\frac{q}{d},-\frac{b}{a},-\frac{qa}{b},\frac{qa}{c},\frac{bc}{a},-\frac{b^2}{a},-\frac{b^2}{a},\frac{qa}{bc},\frac{b^2c}{a},\frac{b^2d}{a^2};q)_\infty
}
\Wpsi88{-a}{\pm b,c,-d,\frac{qa^2}{b^2c},-\frac{qa^2}{b^2d}}{q,\frac{b^2}{a}}\nonumber\\
&&\hspace{-0.2cm}+\frac{(qa,-b,c,\frac{q}{a},\frac{b^2}{a},\frac{bd}{a},\frac{qa}{bd},\frac{qa^2}{b^2c};q)_\infty(q^2,\frac{q^2a}{cd},\frac{b^4cd}{a^3},\frac{qb^2c}{ad},\frac{qb^2d}{ac};q^2)_\infty\vt(-1;q)}
{(\frac{q}{b},\frac{q}{d},-\frac{b^2}{a},\pm \frac{qa}{b},\frac{qa}{c},\frac{qa}{d},\frac{bc}{a},\frac{b^2c}{a},\frac{b^2d}{a},\frac{b^2d}{a},\frac{b^2d}{a^2},\frac{qa}{bc};q)_\infty\vt(-\frac{a}{b};q)}.
\label{8psi88W7sum}
\end{eqnarray}
\end{cor}
\begin{proof}
Start with Theorem~\ref{thmwy} and replace
\[
(d,e,g)\mapsto\left(\frac{qa^2}{b^2c},-b,\frac{qa^2}{b^2f}\right).
\]
Then the resulting very-well-poised ${}_8W_7$ can be summed using \cite[(II.18)]{GaspRah}. Finally replacing $f\mapsto d$ and simplifying the final result completes the proof. 
\end{proof}

\begin{rem}
By taking $b=a$ in Corollaries \ref{cor311} and \ref{cor312}, respectively, it is not difficult to see that these are bilateral extensions of the nonterminating very-well-poised ${}_8W_7$ summations  \cite[(II.16) and (II.18)]{GaspRah} \eqref{GRII.16} and \eqref{GRII.18}, respectively.
\end{rem}

\noindent By equating the Zhang--Zhang and Wei--Yu transformations we can obtain a new transformation for a very-well-poised ${}_8\Psi_8$ series in terms of a ${}_4\psi_4$, an ${}_8W_7$ and two balanced ${}_4\phi_3$ series.
\begin{cor}
Let $0<|q|<1$, $a,b,c,d,e,f,g\in\CCast$ such that
$|q^2a^3|<|bcdefg|$, $|qa|<|fg|$, and such that none of the denominator factors vanish. Then one has the following transformation formula for a very-well-poised ${}_8\Psi_8$, namely 
\begin{eqnarray}
&&\hspace{-0.1cm}\Wpsi88{a}{b,c,d,e,f,g}{q,\frac{q^2a^3}{bcdefg}}\nonumber\\
&&\hspace{0.0cm}=\frac{(qa,\frac{q}{a},\frac{b}{a},\frac{ce}{a},\frac{de}{a},\frac{qa}{cd},\frac{qa}{ce},\frac{qa}{de},\frac{qa}{ef},\frac{qa}{eg},\frac{qa}{fg},\frac{bce}{a},\frac{bde}{a},\frac{q^2a^2}{bcdef},\frac{q^2a^2}{bcdeg};q)_\infty}{(\frac{q}{b},\frac{q}{c},\frac{q}{d},\frac{q}{f},\frac{q}{g},\frac{qa}{c},\frac{qa}{d},\frac{qa}{e},\frac{qa}{f},\frac{qa}{g},\frac{bc}{a},\frac{bd}{a},\frac{qa}{cde},\frac{bcde^2}{qa^2},\frac{q^2a^3}{bcde^2fg};q)_\infty}
\qpsi44{e,\frac{bcde^2}{qa^2},\frac{bcdef}{qa^2},\frac{bcdeg}{qa^2}}{\frac{bce}{a},\frac{bde}{a},\frac{cde}{a},\frac{bcde^2fg}{qa^3}}{q,q}\nonumber\\*
&&\hspace{0.0cm}-\frac{(q,c,d,e,qa,\frac{q}{a},\frac{b}{a},\frac{qb}{c},\frac{qb}{d},\frac{qb}{e},\frac{qa}{bf},\frac{qa}{bg},\frac{qa}{fg},\frac{q^2a^2}{cdef},\frac{q^2a^2}{cdeg};q)_\infty\vt(\frac{qa^2}{bcde};q)}{(\frac{q}{b},\frac{q}{f},\frac{q}{g},\frac{qa}{c},\frac{qa}{d},\frac{qa}{e},\frac{qa}{f},\frac{qa}{g},\frac{bc}{a},\frac{bd}{a},\frac{be}{a},\frac{q^2ab}{cde},\frac{q^2a^3}{bcdefg};q)_\infty\vt(\frac{a}{b},\frac{cde}{a};q)}
\Whyp87{\frac{qab}{cde}}{\frac{qa}{cd},\frac{qa}{ce},\frac{qa}{de},\frac{bf}{a},\frac{bg}{a}}{q,\frac{qa}{fg}}\nonumber\\*
&&\hspace{0.0cm}+
\frac{(q,qa,\frac{q}{a},\frac{b}{a},\frac{qe}{f},\frac{qe}{g},\frac{qa}{cd},\frac{qa}{ce},\frac{qa}{de},\frac{qa}{ef},\frac{qa}{eg},\frac{bcde}{a^2},\frac{q^2a^3}{bcdefg};q)_\infty\vt(\frac{bce}{a},\frac{bde}{a};q)}{(\frac{q}{b},\frac{q}{c},\frac{q}{d},\frac{q}{e},\frac{q}{f},\frac{q}{g},\frac{qa}{c},\frac{qa}{d},\frac{qa}{e},\frac{qa}{f},\frac{qa}{g},\frac{bc}{a},\frac{bd}{a},\frac{be}{a},\frac{q^2a^3}{bcdefg};q)_\infty\vt(\frac{qa^2}{bcde^2};q)}\qhyp43{\frac{be}{a},\frac{ce}{a},\frac{de}{a},\frac{qa}{fg}}{\frac{qe}{f},\frac{qe}{g},\frac{bcde}{a^2}}{q,q}\nonumber\\*
&&\hspace{0.0cm}+
\frac{(q,e,qa,\frac{q}{a},\frac{b}{a},\frac{ce}{a},\frac{de}{a},\frac{qa}{cd},\frac{qa}{ce},\frac{qa}{de},\frac{q^2a^2}{befg},\frac{q^2a^2}{cefg},\frac{q^2a^2}{defg};q)_\infty\vt(\frac{bcdef}{qa^2},\frac{bcdeg}{qa^2};q)}{(\frac{q}{b},\frac{q}{c},\frac{q}{d},\frac{q}{f},\frac{q}{g},\frac{qa}{c},\frac{qa}{d},\frac{qa}{e},\frac{qa}{f},\frac{qa}{g},\frac{bc}{a},\frac{bd}{a},\frac{q^2a^3}{bcdefg};q)_\infty\vt(\frac{cde}{a},\frac{bcde^2fg}{q^2a^3};q)}\qhyp43{\frac{qa}{ef},\frac{qa}{eg},\frac{qa}{fg},\frac{q^2a^3}{bcdefg}}{\frac{q^2a^2}{befg},\frac{q^2a^2}{cefg},\frac{q^2a^2}{defg}}{q,q}.
\label{bigtrans88}
\end{eqnarray}
\label{thm312}
\end{cor}
\begin{proof}
First equate \eqref{8psi84psi4} and \eqref{8psi88W7}, then making the replacements
\[
(a,f,g)\mapsto\left(\frac{bcde}{qa},\frac{qa^2f}{bcde},\frac{qa^2g}{bcde}\right),
\]
and solving for the ${}_8\Psi_8$ completes the proof.
\end{proof}

\begin{rem}
The $b\to a$ limit of Corollary~\ref{thm312} is equivalent to
Gasper and Rahman \cite[(III.23)]{GaspRah}.     
\end{rem}

\section{Very-well-poised bilateral transformations as limiting cases}
\label{sec:lim}
In this section we consider transformations of ${}_{8-p}\Psi_{8-p}^p$ series, for $1\le p\le 6$, obtained as limiting cases
of transformations for the ${}_8\Psi_8$ series presented above. First we consider the case $p=1$, namely transformations
for the ${}_7\Psi_7^1$ series.

\subsection{Transformations for the ${}_7\Psi_7^1$ series}
\noindent Starting with 
\eqref{psi888W7}
leads to a nice corollary as a transformation of a very-well-poised ${}_7\Psi_7^1$.
Note that in the following transformation, while the left-hand side is manifestly symmetric in the five variables $b,c,d,e,f$, the right-hand side satisfies the same symmetry (which otherwise is not obvious).
\begin{cor}
\label{cor42}
Let $0<|q|<1$, $a,b,c,d,e,f\in\CCast$, $|q|<\max\{|de/a|,|df/a|,|cd/a|\}<1$. 
Then one has the following transformation for a very-well-poised ${}_7\Psi_7^1$ series
\begin{eqnarray}
&&\hspace{-1.2cm}
\qpWpsi771{a}{b,c,d,e,f}{q,\frac{q^2a^3}{bcdef}}
\nonumber\\[-0.0cm]
&&\hspace{0cm}
=
\frac{(q,qa,\frac{q}{a},e,f,\frac{e}{a},\frac{f}{a},\frac{qa}{de},\frac{qa}{df};q)_\infty}{(\frac{q}{b},\frac{q}{c},\frac{q}{d},\frac{qa}{b},\frac{qa}{c},\frac{qa}{d},\frac{ef}{a};q)_\infty} 
\II{e;f}
\frac{(\frac{qe}{b},\frac{qe}{c},\frac{qa}{be},\frac{qa}{ce};q)_\infty}{
(
\frac{e}{a},\frac{f}{e},\frac{qa}{e},\frac{qe}{f};q)_\infty\vt(e;q)}\qhyp32{\frac{qa}{bc},\frac{de}{a},\frac{ef}{a}}{\frac{qe}{b},\frac{qe}{c}}{q,\frac{qa}{df}}\label{psi783phi2}
\\
&&\hspace{-0.0cm}
=
\frac{(qa,\frac{q}{a},\frac{qa}{be},\frac{qa}{bf},\frac{qa}{cd},\frac{qa}{ef};q)_\infty}
{(\frac{q}{c},\frac{q}{d},\frac{qa}{b},\frac{qa}{e},\frac{qa}{f},\frac{qa}{bef};q)_\infty}
\qpsi33{b,e,f}{\frac{qa}{c},\frac{qa}{d},\frac{bef}{a}}{q,q}\nonumber\\*
&&\hspace{1cm}+\frac{(q,qa,\frac{q}{a},b,e,f,\frac{qa}{cd},\frac{q^2a^2}{bcef},\frac{q^2a^2}{bdef};q)_\infty}{(\frac{q}{c},\frac{q}{d},\frac{qa}{b},\frac{qa}{c},\frac{qa}{d},\frac{qa}{e},\frac{qa}{f};q)_\infty\vt(\frac{bef}{qa}
;q)
}\qhyp32{\frac{qa}{be},\frac{qa}{bf},\frac{qa}{ef}}{\frac{q^2a^2}{bcef},\frac{q^2a^2}{bdef}}{q,q}
\label{7psi83psi3}\\*
&&\hspace{0.0cm}
=\frac{(qa,\frac{q}{a},\frac{b}{a},\frac{qa}{cd},\frac{qa}{ce},\frac{qa}{de},\frac{bcd}{a},\frac{bce}{a},\frac{bde}{a},\frac{q^2a^2}{bcdef};q)_\infty}{(\frac{q}{f},\frac{qa}{c},\frac{qa}{d},\frac{qa}{e},\frac{bc}{a},\frac{bd}{a},\frac{be}{a},\frac{qa}{cde},\frac{bcde}{a},\frac{q^2a}{bcde};q)_\infty}\qpWpsi771{\frac{bcde}{qa}}{b,c,d,e,\frac{bcdef}{qa^2}}{q,\frac{bcde}{af}}\nonumber\\*
&&\hspace{1cm}+\frac{(q,qa,\frac{q}{a},c,d,e,\frac{qb}{c},\frac{qb}{d},\frac{qa}{bf},\frac{qa}{ef},\frac{bcde}{qa^2},\frac{q^2a^2}{bcde};q)_\infty}
{(\frac{q}{b},\frac{q}{f},\frac{qa}{b},\frac{qa}{c},\frac{qa}{d},\frac{qa}{e},\frac{qa}{f},\frac{bc}{a},\frac{bd}{a},\frac{be}{a},\frac{cde}{qa},\frac{q^2a}{cde};q)_\infty}\qhyp32{\frac{be}{a},\frac{bf}{a},\frac{qa}{cd}}{\frac{qb}{c},\frac{qb}{d}}{q,\frac{qa}{ef}}\label{7psi73phi2}\\
&&\hspace{0.0cm}=\frac{\frac{a}{cd}(q,qa,\frac{q}{a},b,d,\frac{b}{a},\frac{c}{a},\frac{qc}{b},\frac{qc}{d},\frac{qa}{bd},\frac{qa}{be},\frac{qa}{cd},\frac{qa}{ce},\frac{qa}{cf},\frac{qa}{ef};q)_\infty}{(qd,\frac{1}{d},\frac{q}{c},\frac{q}{e},\frac{q}{f},\frac{a}{c},\frac{b}{c},\frac{qa}{d},\frac{qa}{e},\frac{qa}{f},\frac{qc}{a},\frac{qc}{b},\frac{bc}{a},\frac{qa}{bd},\frac{qa}{be};q)_\infty}\qhyp32{\frac{ce}{a},\frac{cf}{a},\frac{qa}{bd}}{\frac{qc}{b},\frac{qc}{d}}{q,\frac{qa}{ef}}\nonumber\\*
&&\hspace{1cm}-\frac{\frac{a}{cd}(q,qa,\frac{q}{a},c,d,\frac{c}{a},\frac{b}{a},\frac{qb}{c},\frac{qb}{d},\frac{qa}{bd},\frac{qa}{be},\frac{qa}{bf},\frac{qa}{cd},\frac{qa}{ce},\frac{qa}{cd};q)_\infty}{(qd,\frac{1}{d},\frac{q}{b},\frac{q}{e},\frac{q}{f},\frac{a}{b},\frac{b}{c},\frac{qa}{d},\frac{qa}{e},\frac{qa}{f},\frac{qb}{a},\frac{qc}{b},\frac{bc}{a},\frac{qa}{cd},\frac{qa}{ce};q)_\infty}\qhyp32{\frac{be}{a},\frac{bf}{a},\frac{qa}{cd}}{\frac{qb}{c},\frac{qb}{d}}{q,\frac{qa}{ef}}.
\label{7psi12phi32}
\end{eqnarray}
\end{cor}
\begin{proof}
Start with \eqref{psi888W7} and take the limit as $e\to\infty$. This converts the very-well-poised ${}_8\Psi_8$ to the very-well-poised  ${}_7\Psi_7^1$  given on the left-hand side of \eqref{psi783phi2}. Taking the limit of both ${}_8W_7$'s on the right-hand side of  \eqref{psi888W7} 
as $e\to\infty$ converts both ${}_8W_7$'s to ${}_7W_7$'s with one vanishing denominator element.
These can be converted to ${}_3\phi_2$'s using 
the transformation \eqref{W761tran}, establishing \eqref{psi783phi2}.
Starting with \eqref{8psi84psi4} instead and letting $b\mapsto q^{-n}a$ (or $c\mapsto q^{-n}a$),
the coefficient of the second series on the right-hand side vanishes.
After taking the limit $n\to\infty$, the identity becomes \eqref{7psi83psi3} up to a simple substitution of variables.
To get \eqref{7psi73phi2} start with \eqref{8psi88W7}, let $g\to\infty$ and apply \eqref{W761tran}.
Finally, we sketch how to get \eqref{7psi12phi32}: Apply \eqref{psi783phi2} to \eqref{7psi73phi2} to obtain a transformation 
of a  very-well-poised  ${}_7\Psi_7^1$ series into a sum of three  ${}_3\phi_2$ series of which two ${}_3\phi_2$ series
are the same but have different prefactors. The two different prefactors of the same ${}_3\phi_2$ series can be
combined by virtue of the addition formula for the modified theta function in \eqref{tfa}.
The third ${}_3\phi_2$ series can finally be transformed using \eqref{eq:32-1}, completing the
sketch of proof of \eqref{7psi12phi32} (while leaving the full details to the reader).
\end{proof}
\begin{cor}
\label{cor4x}
Let $0<|q|<1$, $a,b,c,d,e,f\in\CCast$, $|q|<\max\{|de/a|,|df/a|,|cd/a|\}<1$. 
Then one has the following transformation for a very-well-poised ${}_7\Psi_7^1$ 
\begin{eqnarray}
&&\hspace{-1.2cm}
\qpWpsi771{a}{b,c,d,e,f}{q,\frac{q^2a^3}{bcdef}}=\frac{(qa,\frac{q}{a},\frac{qa}{bc},\frac{qa}{bd},\frac{qa}{cd},\frac{qa}{ef};q)_\infty}{(\frac{q}{b},\frac{q}{c},\frac{q}{d},\frac{qa}{e},\frac{qa}{f},\frac{qa^2}{bcd};q)_\infty}\qpsi33{e,f,\frac{qa^2}{bcd}}{\frac{qa}{b},\frac{qa}{c},\frac{qa}{d}}{q,\frac{qa}{ef}}\nonumber\\*
&&\hspace{0.0cm}+\frac{(q,qa,\frac{q}{a},\frac{b}{a},\frac{c}{a},\frac{d}{a},\frac{qa}{ef},\frac{q^2a^2}{bcde},\frac{q^2a^2}{bcdf};q)_\infty}{(\frac{q}{b},\frac{q}{c},\frac{q}{d},\frac{q}{e},\frac{q}{f},\frac{qa}{e},\frac{qa}{f};q)_\infty\vt(\frac{bcd}{qa^2};q)}\qhyp32{\frac{qa}{bc},\frac{qa}{bd},\frac{qa}{cd}}{\frac{q^2a^2}{bcde},\frac{q^2a^2}{bcdf}}{q,q}.
\label{cor42eq}
\end{eqnarray}
\end{cor}
\begin{proof}
This result is obtained by starting with \eqref{8psi84psi4} and taking the limit as $g\to\infty$. The final term vanishes and this completes the proof.
\end{proof}
By starting with Theorem~\ref{thm312} and taking the limit as $g\to \infty$ one obtains a new transformation
for a ${}_7\Psi_7^1$ in terms of a ${}_3\psi_3$ and two ${}_3\phi_2$'s.
\begin{cor}
Let $0<|q|<1$, $a,b,c,d,e,f\in\CCast$, $|qa|<|ef|$. 
Then one has the following transformation for a very-well-poised ${}_7\Psi_7^1$ series:
\begin{eqnarray}
&&\hspace{-0.1cm}
\qpWpsi771{a}{b,c,d,e,f}{q,\frac{q^2a^3}{bcdef}}
=\frac{(qa,\frac{q}{a},\frac{b}{a},\frac{ce}{a},\frac{de}{a},\frac{qa}{cd},\frac{qa}{ce},\frac{qa}{de},\frac{qa}{ef},\frac{bce}{a},\frac{bde}{a},\frac{q^2a^2}{bcdef};q)_\infty}{(\frac{q}{b},\frac{q}{c},\frac{q}{d},\frac{q}{f},\frac{qa}{c},\frac{qa}{d},\frac{qa}{e},\frac{qa}{f},\frac{bc}{a},\frac{bd}{a},\frac{qa}{cde},\frac{bcde^2}{qa^2};q)_\infty}\qpsi33{e,\frac{bcde^2}{qa^2},\frac{bcdef}{qa^2}}{\frac{bce}{a},\frac{bde}{a},\frac{cde}{a}}{q,\frac{qa}{ef}}\nonumber\\*
&&\hspace{2.0cm}-\frac{(q,c,d,e,qa,\frac{q}{a},\frac{b}{a},\frac{qb}{c},\frac{qb}{d},\frac{qa}{bf},\frac{qa}{ef};q)_\infty\vt(\frac{qa^2}{bcde};q)}{(\frac{q}{b},\frac{q}{f},\frac{qa}{c},\frac{qa}{d},\frac{qa}{e},\frac{qa}{f},\frac{bc}{a},\frac{bd}{a},\frac{be}{a};q)_\infty\vt(\frac{a}{b},\frac{cde}{a};q)}\qhyp32{\frac{be}{a},\frac{bf}{a},\frac{qa}{cd}}{\frac{qb}{c},\frac{qb}{d}}{q,\frac{qa}{ef}}\nonumber\\*
&&\hspace{2.2cm}+\frac{(q,qa,\frac{q}{a},\frac{b}{a},\frac{qe}{f},\frac{qa}{cd},\frac{qa}{ce},\frac{qa}{de},\frac{qa}{ef},\frac{bcde}{a^2};q)_\infty\vt(\frac{bce}{a},\frac{bde}{a};q)}{(\frac{q}{b},\frac{q}{c},\frac{q}{d},\frac{q}{e},\frac{q}{f},\frac{qa}{c},\frac{qa}{d},\frac{qa}{e},\frac{qa}{f},\frac{bc}{a},\frac{bd}{a},\frac{be}{a};q)_\infty\vt(\frac{qa^2}{bcde^2};q)}
\qhyp32{\frac{be}{a},\frac{ce}{a},\frac{de}{a}}{\frac{qe}{f},\frac{bcde}{a^2}}{q,q}.
\label{7Psi713psi3trans}
\end{eqnarray}
\label{cor45}
\end{cor}
\begin{proof}
Start with Theorem~\ref{thm312} and take the limit as $g\to \infty$. In the limit the last term on the right-hand side of \eqref{bigtrans88} vanishes. The ${}_4\psi_4$ becomes the ${}_3\psi_3$ and the ${}_8W_7$ becomes a ${}_7W_6^1$ which can be written as a ${}_3\phi_2$ using \eqref{W761tran}. This completes the proof.
\end{proof}

\subsection{Transformations for the ${}_6\Psi_6^2$ series}

Note that in the following transformation, while the left-hand side is manifestly symmetric in the four variables $b,c,d,e$, the right-hand side satisfies the same symmetry (which otherwise is not obvious).

\begin{cor}
\label{cor52}
Let $0<|q|<1$, $b,c,d,e\in\CCast$. Then
\begin{eqnarray}
&&\hspace{-0.0cm}
\qpWpsi{6}{6}{2}{a}{b,c,d,e}{q,\frac{q^2a^3}{bcde}}
=\frac{(q,qa,\frac{q}{a},d,e,\frac{qa}{cd},\frac{qa}{ce};q)_\infty}{(\frac{q}{b},\frac{q}{c},\frac{qa}{b},\frac{qa}{c},\frac{de}{a};q)_\infty}\II{d;e}
\frac{(\frac{e}{a},\frac{qd}{b},\frac{qa}{bd};q)_\infty}{(d,\frac{q}{d},\frac{e}{d},\frac{qa}{d},\frac{qd}{e};q)_\infty}\qhyp21{\frac{de}{a},\frac{cd}{a}}{\frac{qd}{b}}{q,\frac{qa}{ce}}\label{6psi62ph1}\\
&&\hspace{0.6cm}=\frac{(qa,\frac{q}{a},\frac{qa}{bd},\frac{qa}{be},\frac{qa}{de};q)_\infty}{(\frac{q}{c},\frac{qa}{b},\frac{qa}{d},\frac{qa}{e},\frac{qa}{bde};q)_\infty}
\qppsi321{b,d,e}{\frac{qa}{c},\frac{bde}{a}}{q,q}
\nonumber\\*&&\hspace{1.6cm}
+
\frac{(q,qa,\frac{q}{a},b,d,e,\frac{q^2a^2}{bcde};q)_\infty}{(\frac{q}{c},\frac{qa}{b},\frac{qa}{c},\frac{qa}{d},\frac{qa}{e};q)_\infty\vt(\frac{bde}{qa}
;q)
}\qphyp311{\frac{qa}{bd},\frac{qa}{be},\frac{qa}{de}}{\frac{q^2a^2}{bcde}}{q,q}\label{6psi63ph1}\\
&&\hspace{0.6cm}
=\frac{(qa,\frac{q}{a},\frac{qa}{bc},\frac{qa}{de};q)_\infty}{(\frac{q}{d},\frac{q}{e},\frac{qa}{b},\frac{qa}{c};q)_\infty}\qpsi22{b,c}{\frac{qa}{d},\frac{qa}{e}}{q,\frac{qa}{bc}}
\label{psi682phi2}\\
&&\hspace{0.6cm}
=\frac{(qa,\frac{q}{a},\frac{qa}{cd},\frac{qa}{ce},\frac{qa}{de},\frac{b}{a},\frac{bcd}{a},\frac{bce}{a},\frac{bde}{a};q)_\infty}{(\frac{q}{f},\frac{qa}{c},\frac{qa}{d},\frac{qa}{e},\frac{bc}{a},\frac{bd}{a},\frac{be}{a},\frac{qa}{cde},\frac{bcde}{a},\frac{q^2a}{bcde};q)_\infty}\qpWpsi{6}{6}{2}{\frac{bcde}{qa}}{b,c,d,e}{q,\frac{b^2c^2d^2e^2}{qa^3}}\nonumber\\*
&&\hspace{1.6cm}+
\frac{(q,qa,\frac{q}{a},c,d,e,\frac{qb}{c};q)_\infty\vt(\frac{bcde}{qa^2};q)}{(\frac{q}{b},\frac{qa}{b},\frac{qa}{c},\frac{qa}{d},\frac{qa}{e},\frac{bd}{a},\frac{be}{a};q)_\infty\vt(\frac{cde}{qa};q)}\qhyp21{\frac{qa}{cd},\frac{qa}{ce}}{\frac{qb}{c}}{q,\frac{bc}{a}}\label{6psi62}\\
&&\hspace{0.6cm}
=\frac{(qa,\frac{q}{a},\frac{b}{a},\frac{de}{a},\frac{qa}{cd},\frac{qa}{ce},\frac{qa}{de},\frac{bcd}{a},\frac{bce}{a};q)_\infty}{(\frac{q}{d},\frac{q}{e},\frac{q}{f},\frac{qa}{c},\frac{qa}{d},\frac{qa}{e},\frac{bd}{a},\frac{be}{a};q)_\infty\vt(\frac{cde}{a};q)}\qpsi{2}{2}{b,c}{\frac{bcd}{a},\frac{bce}{a}}{q,\frac{de}{a}}\nonumber\\*
&&\hspace{1.6cm}+
\frac{(q,qa,\frac{q}{a},c,d,e,\frac{qb}{c};q)_\infty\vt(\frac{bcde}{qa^2};q)}{(\frac{q}{b},\frac{qa}{b},\frac{qa}{c},\frac{qa}{d},\frac{qa}{e},\frac{bd}{a},\frac{be}{a};q)_\infty\vt(\frac{cde}{qa};q)}\qhyp21{\frac{qa}{cd},\frac{qa}{ce}}{\frac{qb}{c}}{q,\frac{bc}{a}},\label{6psi62s}
\end{eqnarray}
provided the series converge.
\end{cor}

\begin{proof}[Proof of Corollary~\ref{cor52}]
Start with \eqref{psi783phi2} and take the limit as $d\to\infty$, then replacing $f\mapsto d$ converts the ${}_3\phi_2$'s to ${}_2\phi_2$'s. Then one can use \eqref{rel2122} twice to write the ${}_2\phi_2$'s in terms of ${}_2\phi_1$'s, hereby establishing \eqref{6psi62ph1}.
By taking the limit as $c\to\infty$ of \eqref{7psi83psi3} followed by $f\mapsto c$, we obtain a bilateral extension of \eqref{2phi3phi2dzero} (after taking $c=a$), establishing \eqref{6psi63ph1}.
One can use \eqref{Bailey} to write the left-hand side as a ${}_2\psi_2$, establishing \eqref{psi682phi2}.
Equation \eqref{6psi62} can be obtained by a direct limit from \eqref{7psi73phi2}, while Equation \eqref{6psi62s}
is obtained by combining \eqref{6psi62} with \eqref{psi682phi2}.
This completes the proof.
\end{proof}

By starting with Corollary~\ref{cor45} and taking the limit as $f\to \infty$ one obtains a new transformation for a ${}_6\Psi_6^2$ in terms of a ${}_2\psi_3$, one ${}_2\phi_1$ and a ${}_3\phi_1^1$.
\begin{cor}
Let $0<|q|<1$, $a,b,c,d,e,f\in\CCast$, $|q|<\max\{|de/a|,|df/a|,|cd/a|\}<1$. 
Then one has the following transformation for a very-well-poised ${}_6\Psi_6^2$ 
\begin{eqnarray}
&&\hspace{-0.1cm}
\qpWpsi662{a}{b,c,d,e}{q,\frac{q^2a^3}{bcde}}
=\frac{(qa,\frac{q}{a},\frac{b}{a},\frac{ce}{a},\frac{de}{a},\frac{qa}{cd},\frac{qa}{ce},\frac{qa}{de},\frac{bce}{a},\frac{bde}{a};q)_\infty}{(\frac{q}{b},\frac{q}{c},\frac{q}{d},\frac{qa}{c},\frac{qa}{d},\frac{qa}{e},\frac{bc}{a},\frac{bd}{a},\frac{qa}{cde},\frac{bcde^2}{qa^2};q)_\infty}\qpsi23{e,\frac{bcde^2}{qa^2}}{\frac{bce}{a},\frac{bde}{a},\frac{cde}{a}}{q,\frac{bcd}{a}}\nonumber\\*
&&\hspace{2.0cm}-\frac{(q,c,d,e,qa,\frac{q}{a},\frac{b}{a},\frac{qb}{c},\frac{qa}{de};q)_\infty\vt(\frac{qa^2}{bcde};q)}{(\frac{q}{b},\frac{qa}{c},\frac{qa}{d},\frac{qa}{e},\frac{bc}{a},\frac{bd}{a},\frac{be}{a};q)_\infty\vt(\frac{a}{b},\frac{cde}{a};q)}\qhyp21{\frac{bd}{a},\frac{be}{a}}{\frac{qb}{c}}{q,\frac{qa}{de}}\nonumber\\*
&&\hspace{2.2cm}+\frac{(q,qa,\frac{q}{a},\frac{b}{a},\frac{qa}{cd},\frac{qa}{ce},\frac{qa}{de},\frac{bcde}{a^2};q)_\infty\vt(\frac{bce}{a},\frac{bde}{a};q)}{(\frac{q}{b},\frac{q}{c},\frac{q}{d},\frac{q}{e},\frac{qa}{c},\frac{qa}{d},\frac{qa}{e},\frac{bc}{a},\frac{bd}{a},\frac{be}{a};q)_\infty\vt(\frac{qa^2}{bcde^2};q)}
\qphyp311{\frac{be}{a},\frac{ce}{a},\frac{de}{a}}{\frac{bcde}{a^2}}{q,q}.
\label{cor44eq}
\end{eqnarray}
\end{cor}
\begin{proof}
Start with Corollary~\ref{cor45} and take the limit as $f\to \infty$. The ${}_3\psi_3$ becomes the ${}_2\psi_3$ and the first ${}_3\phi_2$ becomes a ${}_2\phi_2$ which can be written as a ${}_2\phi_1$ using \eqref{rel2122}. The last ${}_3\phi_2$ becomes a ${}_3\phi_1^1$ which completes the proof.
\end{proof}

\subsection{Transformations for the ${}_5\Psi_5^3$ series}

Note that in the following transformation, while the left-hand side is manifestly symmetric in the three variables $b,c,d$, the right-hand side satisfies the same symmetry (which otherwise is not obvious).

\begin{cor}
\label{cor54}
Let $0<|q|<1$, $a,b,c,d\in\CCast$. Then
\begin{eqnarray}
&&\hspace{-0.5cm}
\qpWpsi553{a}{b,c,d}{q,\frac{q^2a^3}{bcd}}\
=\frac{(qa,\frac{q}{a},\frac{qa}{bc};q)_\infty}{(\frac{q}{b},\frac{q}{c},\frac{qa}{d};q)_\infty}\qpsi12{d}{\frac{qa}{b},\frac{qa}{c}}{q,\frac{qa}{d}}\nonumber\\
&&\hspace{0.3cm}=\frac{(qa,\frac{q}{a},\frac{qa}{bd};q)_\infty}{(\frac{q}{c},\frac{qa}{b},\frac{qa}{d};q)_\infty}\qppsi211{b,d}{\frac{qa}{c}}{q,\frac{qa}{bd}}\nonumber\\
&&\hspace{0.3cm}=\frac{(qa,\frac{q}{a},\frac{qa}{bc},\frac{qa}{bd},\frac{qa}{cd};q)_\infty}{(\frac{qa}{b},\frac{qa}{c},\frac{qa}{d},\frac{qa}{bcd};q)_\infty}\qppsi312{b,c,d}{\frac{bcd}{a}}{q,q}+\frac{(q,qa,\frac{q}{a},b,c,d;q)_\infty}{(\frac{qa}{b},\frac{qa}{c},\frac{qa}{d},\frac{bcd}{qa},\frac{q^2a}{bcd};q)_\infty}\qphyp302{\frac{qa}{bc},\frac{qa}{bd},\frac{qa}{cd}}{-}{q,q}.
\label{eqcor54}
\end{eqnarray}
\end{cor}
\begin{proof}
Start with Corollary~\ref{cor52}. Taking the limit as $e\to\infty$ of the left-hand side produces the left-hand side of \eqref{eqcor54}. Taking the limit of the ${}_2\psi_1^2$ representation as $b\to\infty$ produces the ${}_1\psi_2$ representation. One can obtain the ${}_2\psi_1^1$ representation from the ${}_1\psi_2$ representation by reversing the order of summation. Taking the limit of the sum of the ${}_3\psi_1^2$ and ${}_3\phi_2$ representation as $\{e,c\}\to\infty$ produces the ${}_2\psi_1^1$ and ${}_3\psi_1^2$ plus ${}_3\phi_0^2$ representations respectively. This completes the proof. 
\end{proof}

\subsection{Transformations for the ${}_4\Psi_4^4$ series}
Note that in the following transformation, while the left-hand side is manifestly symmetric in the two variables $b,c$, the right-hand side satisfies the same symmetry (which otherwise is not obvious).

\begin{cor}
\label{cor55}
Let
$0<|q|<1$, $a,b,c\in\CCast$. Then
\begin{eqnarray}
&&\hspace{-1.0cm}
\qpWpsi444{a}{b,c}{q,\frac{q^2a^3}{bc}}
=
\frac{(qa,\frac{q}{a};q)_\infty}{(\frac{q}{c},\frac{qa}{b};q)_\infty}
\qppsi111{b}{\frac{qa}{c}}{q,\frac{qa}{b}}
=\frac{(qa,\frac{q}{a},\frac{qa}{bc};q)_\infty}{(\frac{q}{b},\frac{q}{c};q)_\infty}\qpsi02{-}{\frac{qa}{b},\frac{qa}{c}}{q,qa}\nonumber\\
&&\hspace{2.6cm}
=\frac{(qa,\frac{q}{a},\frac{qa}{bc};q)_\infty}{(\frac{qa}{b},\frac{qa}{c};q)_\infty}\qppsi202{b,c}{-}{q,\frac{qa}{bc}}.
\label{cor54eq}
\end{eqnarray}
\end{cor}
\begin{proof}
Start with Corollary~\ref{cor54}. Taking the limit as $d\to\infty$ produces the left-hand side of \eqref{cor54eq}. Taking the limit as $c\to\infty$ in the ${}_1\psi_2$ representation and replacing $d\mapsto c$ produces the ${}_1\psi_2$ representation. Taking the limit as $d\to\infty$ in the ${}_1\psi_2$ representation produces the ${}_0\psi_2$ representation. Taking the limit as $c\to\infty$ in the ${}_2\psi_2$ representation and replacing $d\mapsto c$ produces the ${}_2\psi_2$ representation.
\end{proof}

\subsection{Transformations for the ${}_3\Psi_3^5$ series}

\begin{cor}
\label{cor56}
Let $0<|q|<1$, $a,b\in\CCast$. Then
\begin{eqnarray}
&&\hspace{-2.5cm}
\qpWpsi335{a}{b}{q,\frac{q^2a^3}{b}}
=\frac{(qa,\frac{q}{a};q)_\infty}{(\frac{qa}{b};q)_\infty}
\qppsi102{b}{-}{q,\frac{qa}{b}}=\frac{(qa,\frac{q}{a};q)_\infty}{(\frac{q}{b};q)_\infty}\qppsi011{-}{\frac{qa}{b}}{q,qa}.
\label{cor55eq}
\end{eqnarray}
\end{cor}
\begin{proof}
Start with Corollary~\ref{cor55}. Taking the limit as $c\to\infty$ produces the left-hand side of \eqref{cor55eq}. Taking the limit as $c\to\infty$ of the ${}_1\psi_2$ representation produces the ${}_1\psi_2$ representation.  Taking the limit as $b\to\infty$ of the ${}_1\psi_2$ representation produces the ${}_0\psi_2$ representation.
\end{proof}

\subsection{Summation for the ${}_2\Psi_2^6$ series}

The following summation is a restatement of Watson's quintuple product identity \cite[\href{http://dlmf.nist.gov/17.8.E3}{(17.8.3)}]{NIST:DLMF} (see also \cite[Exercise 5.6]{GaspRah}).
\begin{cor}
\label{cor57}
Let $0<|q|<1$, $a\in\CCast$, $a\ne 1$. Then
\begin{eqnarray}
&&\hspace{-2.5cm}
(1-a)\qpWpsi226{a}{-}{q,q^2a^3}
=
\vartheta(a;q)
\qppsi002{-}{-}{q,qa}=
(q^2;q^2)_\infty\vt(a;q^2)\vt(q^2a^2;q^4).
\end{eqnarray}
\end{cor}
\begin{proof}
Start with Corollary~\ref{cor56}, where each series is multiplied by $(1-a)$. Taking the limit as $b\to\infty$ of the ${}_1\psi_2$ or ${}_0\psi_2$ representations produces the ${}_0\psi_2$ representation which is evaluated using the triple-product identity
\eqref{tfdef}.
\end{proof}

\section{Implied bilateral basic transformations and summations}
\label{sec:imp}

In this section we list some implications of the above transformations for
very-well-poised bilateral basic hypergeometric series in terms of
bilateral basic hypergeometric series that in general are not very-well-poised.

\subsection{Transformations for the balanced ${}_4\psi_4$ series}

\noindent Using Theorem~\ref{thm312}, one may obtain a transformation of a balanced ${}_4\psi_4$ in terms a two ${}_8W_7$'s and two balanced ${}_4\phi_3$'s.

\begin{thm}Let $0<|q|<1$, $a,b,c,d,e,f,g,\frac{q^2abcd}{efg}\in\CCast$, $|efg|<|qbcd|$. Then
\begin{eqnarray}
&&\hspace{0.0cm}\qpsi44{a,b,c,d}{e,f,g,\frac{q^2abcd}{efg}}{q,q}\nonumber\\
&&\hspace{0.275cm}=\frac{c}{e}\frac{(q,a,c,\frac{q}{e},\frac{q}{f},\frac{c}{a},\frac{f}{a},\frac{f}{b},\frac{f}{d},\frac{qa}{c},\frac{qe}{f},\frac{qg}{f},\frac{eg}{ab},\frac{eg}{ac},\frac{eg}{ad},\frac{fg}{qc},\frac{q^2c}{ef},\frac{q^2c}{fg},\frac{efg}{qabcd};q)_\infty}{(g,\frac{q}{b},\frac{q}{d},\frac{f}{q},\frac{g}{c},\frac{q^2}{f},\frac{qc}{f},\frac{qc}{g},\frac{qeg}{af},\frac{efg}{qabc},\frac{efg}{qabd},
\frac{efg}{qacd};q)_\infty\,\vt(a,\frac{c}{a},\frac{f}{e},\frac{ef}{c};q)}
\Whyp87{\frac{eg}{af}}{\frac{e}{a},\frac{g}{a},\frac{qb}{f},\frac{qc}{f},\frac{qd}{f}}{q,\frac{efg}{qbcd}}\nonumber\\*
&&\hspace{0.275cm}-\frac{c}{e}\frac{(q,a,c,\frac{q}{e},\frac{q}{f},\frac{c}{a},\frac{e}{a},\frac{e}{b},\frac{e}{d},\frac{qa}{c},\frac{qf}{e},\frac{qg}{e},\frac{eg}{qc},\frac{fg}{ab},\frac{fg}{ac},\frac{fg}{ad},\frac{q^2c}{ef},\frac{q^2c}{eg},\frac{efg}{qabcd};q)_\infty}{(g,\frac{q}{b},\frac{q}{d},\frac{e}{q},\frac{g}{c},\frac{q^2}{e},\frac{qc}{e},\frac{qc}{g},\frac{qfg}{ae},\frac{efg}{qabc},\frac{efg}{qabd},
\frac{efg}{qacd};q)_\infty\vt(a,\frac{c}{a},\frac{f}{e},\frac{ef}{c};q)}
\Whyp87{\frac{fg}{ae}}{\frac{f}{a},\frac{g}{a},\frac{qb}{e},\frac{qc}{e},\frac{qd}{e}}{q,\frac{efg}{qbcd}}
\nonumber\\*
&&\hspace{0.275cm}
-\frac{(q,c,\frac{q}{e},\frac{q}{f},\frac{q}{g},\frac{qc}{a},\frac{qc}{b},\frac{qc}{d},\frac{efg}{qabcd};q)_\infty}{\vt(qc;q)(\frac{q}{a},\frac{q}{b},\frac{q}{d},\frac{qc}{e},\frac{qc}{f},\frac{qc}{g},\frac{efg}{qabd};q)_\infty}\qhyp43{\frac{qc}{e},\frac{qc}{f},\frac{qc}{g},\frac{efg}{qabd}}{\frac{qc}{a},\frac{qc}{b},\frac{qc}{d}}{q,q}\nonumber\\*
&&\hspace{0.275cm}-\frac{(q,a,b,c,d,\frac{efg}{qabcd},\frac{e^2fg}{qabcd},\frac{ef^2g}{qabcd},
\frac{efg^2}{qabcd};q)_\infty}
{\vt(\frac{efg}{abcd};q)(e,f,g,\frac{efg}{qabc},\frac{efg}{qabd},\frac{efg}{qacd},\frac{efg}{qbcd};q)_\infty}\qhyp43{\frac{efg}{qabc},\frac{efg}{qabd},\frac{efg}{qacd},\frac{efg}{qbcd}}{\frac{e^2fg}{qabcd},\frac{ef^2g}{qabcd},
\frac{efg^2}{qabcd}}{q,q}.
\label{4psi4eq}
\end{eqnarray}
\label{thm43}
\end{thm}
\begin{proof}
Start with \eqref{bigtrans88}, then notice that the ${}_8\Psi_8$ is invariant under the interchange $b\leftrightarrow c$. Then equating the right-hand sides under this replacement, making the replacements
\[
(a,b,c,d,e,f,g)\mapsto\left(\frac{aefg}{q^2c^2},\frac{fg}{qc},\frac{eg}{qc},\frac{ef}{qc},a,\frac{ab}{c},\frac{ad}{c}\right),
\]
and solving for the ${}_4\psi_4$ 
completes the proof. 
\end{proof}

One can obtain the following summation theorem by starting with Theorem 
\ref{thm22}, and setting $(d,f)=(1,a)$.
\begin{cor} Let $0<|q|<1$, $a,b,c,d,e\in\CCast$. Then one has the following summation theorem for a ${}_4\psi_4$ and two ${}_4\phi_3$'s namely 
\begin{eqnarray}
&&\frac{(\frac{qa}{d},\frac{qa}{e};q)_\infty}{(\frac{q}{bc},\frac{qa^2}{de};q)_\infty}\qpsi44{a,b,c,\frac{qa^2}{de}}{qa,bc,\frac{qa}{d},\frac{qa}{e}}{q,q}+
\frac{(q,\frac{1}{a},\frac{d}{a},\frac{e}{a},\frac{q^2a}{de},\frac{q^2a^2}{bde},\frac{q^2a^2}{cde};q)_\infty}{(\frac{q}{a},\frac{q}{b},\frac{q}{c},\frac{qa}{de},\frac{de}{qa^2},\frac{q^2a^2}{de},\frac{q^2a^2}{bcde};q)_\infty}\qhyp43{\frac{qa}{d},\frac{qa}{e},\frac{qa}{de},\frac{q^2a^2}{bcde}}{\frac{q^2a}{de},\frac{q^2a^2}{bde},\frac{q^2a^2}{cde}}{q,q}\nonumber\\*
&&\hspace{2.5cm}+\frac{(q,a,b,c,\frac{q^2a}{bc},\frac{q^2a}{bcd},\frac{q^2a}{bce};q)_\infty}{(qa,\frac{q}{b},\frac{q}{c},\frac{qa}{bc},\frac{q^2}{bc},\frac{q^2a^3}{bcde};q)_\infty}\qhyp43{\frac{q}{b},\frac{q}{c},\frac{qa}{bc},\frac{q^2a^2}{bcde}}{\frac{q^2a}{bc},\frac{q^2a}{bcd},\frac{q^2}{bce}}{q,q}=
\frac{(q,q,\frac{q}{d},\frac{q}{e},\frac{qa}{b},\frac{qa}{c};q)_\infty}
{(qa,\frac{q}{a},\frac{q}{b},\frac{q}{c},\frac{qa}{bc},\frac{qa}{de};q)_\infty}.
\end{eqnarray}
\end{cor}
\begin{proof}
Start with Theorem~\ref{thm22} and setting $d=1$ and $f=a$ then the ${}_8\Psi_8$ on the left-hand side becomes unity. Setting $g\mapsto d$ and simplifying completes the proof.
\end{proof}

\begin{rem} 
Note that there exists general transformations of a ${}_r\psi_r$ given in \cite[Section 5.4]{GaspRah} which are due to Slater \cite{Slater52}. 
For instance, taking
\cite[(5.4.5)]{GaspRah}, for $r=4$, this expansion is given by
\begin{eqnarray}
&&\hspace{-3.5cm}\qpsi44{a,b,c,d}{e,f,g,h}{q,z}=
\frac{(q;q)_\infty}{\vt(\frac{abcdz}{efgh};q)(\frac{q}{a},\frac{q}{b},\frac{q}{c},\frac{q}{d};q)_\infty}\nonumber\\*
&&\hspace{-2cm}
\times\II{e;f,g,h}\!\!\!\!\frac{q}{e}\frac{\vt(\frac{abcdz}{qfgh};q)(\frac{q}{f},\frac{q}{g},\frac{q}{h},\frac{e}{a},\frac{e}{b},\frac{e}{c},\frac{e}{d};q)_\infty}{(e,\frac{e}{f},\frac{e}{g},\frac{e}{h};q)_\infty}\qhyp43{\frac{qa}{e},\frac{qb}{e},\frac{qc}{e},\frac{qd}{e}}{\frac{qf}{e},\frac{qg}{e},\frac{qh}{e}}{q,z},\label{eq:Slater4}
\end{eqnarray}
such that $0<|q|<1$, $|z|<1$, $a,b,c,d,e,f,g,h\in\CCast$ with no vanishing denominator factors. However, these expansions do not seem useful for the purpose of obtaining summations or simplified transformations by taking limits. This is because the factors in Slater's transformations involve modified theta functions that contain all parameters, such as the factor $\vt(\frac{abcdz}{qfgh};q)$
in \eqref{eq:Slater4}, which makes it impossible to take confluent limits.
\end{rem}

\subsection{Transformations for the special ${}_3\psi_3$ series}

By starting with \eqref{cor42eq} and interchanging $d,f$ and then $e,f$ and then comparing the right-hand sides, we can obtain two transformations for a special ${}_3\psi_3$ series.
\begin{thm}Let $0<|q|<1$, $a,b,c,d,e,f\in\CCast$ such that $|def|<|qabc|$. Then
\begin{eqnarray}
&&\hspace{-0.0cm}\qpsi33{a,b,c}{d,e,f}{q,\frac{def}{qabc}}
=
\frac{(c,\frac{f}{a},\frac{d}{b},\frac{e}{b},\frac{q^2c}{de},\frac{def}{qbc};q)_\infty}{(f,\frac{q}{b},\frac{qc}{d},\frac{qc}{e},\frac{def}{qabc},\frac{de}{qb};q)_\infty}\qpsi33{a,\frac{de}{qb},\frac{de}{qc}}{d,e,\frac{def}{qbc}}{q,\frac{f}{a}}\nonumber\\*
&&\hspace{0.5cm}+\frac{(q,c,\frac{q}{d},\frac{q}{e},\frac{qc}{b};q)_\infty}{(\frac{q}{a},\frac{q}{b},\frac{qc}{d},\frac{qc}{e},\frac{qc}{f};q)_\infty}\left(\frac{(\frac{f}{a},\frac{de}{ab};q)_\infty\vt(\frac{def}{qbc};q)}{(f,\frac{def}{qabc};q)_\infty\vt(\frac{qb}{de};q)}\qhyp32{\frac{d}{b},\frac{e}{b},\frac{qc}{f}}{\frac{qc}{b},\frac{de}{ab}}{q,q}\!-\!
\frac{(\frac{q}{f},\frac{qc}{a};q)_\infty}{\vt(\frac{1}{c};q)}\qhyp32{\frac{qc}{d},\frac{qc}{e},\frac{qc}{f}}{\frac{qc}{a},\frac{qc}{b}}{q,q}\right)\nonumber
\\&&\label{firstpsi33}\\
&&\hspace{0.4cm}=
\frac{(c,\frac{d}{a},\frac{d}{b},\frac{q^2c}{de},\frac{q^2c}{df},\frac{def}{qac},\frac{def}{qbc};q)_\infty}{(e,f,\frac{q}{a},\frac{q}{b},\frac{qc}{e},\frac{qc}{f},\frac{d^2ef}{q^2abc};q)_\infty}\qpsi33{\frac{de}{qc},\frac{df}{qc},\frac{d^2ef}{q^2abc}}{d,\frac{def}{qac},\frac{def}{qbc}}{q,\frac{qc}{d}}\nonumber\\*
&&\hspace{2.05cm}+\frac{(q,c,\frac{q}{d};q)_\infty}{(\frac{q}{a},\frac{q}{b},\frac{qc}{e},\frac{qc}{f};q)_\infty}\Biggl(\frac{(\frac{de}{ab},\frac{df}{ab},\frac{def}{qac},\frac{def}{qbc},\frac{q^2ac}{def},\frac{q^2bc}{def};q)_\infty}{(e,f,\frac{def}{qabc};q)_\infty\vt(\frac{q^2abc}{d^2ef};q)}\qhyp32{\frac{d}{a},\frac{d}{b},\frac{def}{qabc}}{\frac{de}{ab},\frac{df}{ab}}{q,q}\nonumber\\*
&&\hspace{7.5cm}-
\frac{(\frac{q}{e},\frac{q}{f},\frac{qc}{a},\frac{qc}{b};q)_\infty}{(\frac{qc}{d};q)_\infty\vt(\frac{1}{c};q)}\qhyp32{\frac{qc}{d},\frac{qc}{e},\frac{qc}{f}}{\frac{qc}{a},\frac{qc}{b}}{q,q}\Biggr).\label{secpsi33}
\end{eqnarray}
\end{thm}
\begin{proof}
Notice that \eqref{cor42eq} is invariant under the interchange of the $b,c,d,e,f$ variables. Equate the right-hand side of \eqref{cor42eq} with the same right-hand side with $d,f$ interchanged. 
Making the variable replacements
\begin{equation}
(a,b,c,d,e,f)\mapsto(\frac{def}{q^2c},\frac{ef}{qc},\frac{df}{qc},\frac{de}{qc},a,b)
\label{varrep}
\end{equation}
produces \eqref{firstpsi33}. Equating the right-hand side of \eqref{cor42eq} with $d,f$ interchanged with the same right-hand side with $e,f$ interchanged and making the identical variable replacements \eqref{varrep} produces \eqref{secpsi33}, which completes the proof.
\end{proof}
\begin{rem}
If one sets $d=q$ in \eqref{firstpsi33}, \eqref{secpsi33}, or $e=q$ in \eqref{firstpsi33}, then the second terms on the right-hand sides vanish. This then produces the two-term nonterminating transformations for a special ${}_3\phi_1$ \cite[(III.9), (III.10)]{GaspRah}. On the other hand, if one sets $e=q$ in \eqref{secpsi33} or $f=q$ in \eqref{firstpsi33}, \eqref{secpsi33}, then the special ${}_3\psi_3$ on the left-hand side becomes a special ${}_3\phi_2$ and the ${}_3\psi_3$ on the right-hand sides become a specialized ${}_3\psi_3$. For instance, for the $e=q$ in \eqref{secpsi33} case, solving for this specialized ${}_3\psi_3$ replacing variables produces the following transformation
\begin{eqnarray}
&&\hspace{-3.5cm}\qpsi33{a,b,c}{d,e,\frac{q^2bc}{de}}{q,\frac{q}{a}}=\frac{(q,c,\frac{d}{a},\frac{d}{b},\frac{e}{b},\frac{qb}{a},\frac{q^2c}{de};q)_\infty}{(e,\frac{q}{a},\frac{q}{b},\frac{d}{a},\frac{qc}{e},\frac{de}{qb},\frac{q^2bc}{de};q)_\infty}\qhyp32{\frac{d}{a},\frac{qb}{e},\frac{de}{qc}}{d,\frac{qb}{a}}{q,\frac{qc}{d}}\nonumber\\*
&&\hspace{0.3cm}-\frac{(q,c,\frac{q}{d},\frac{q}{e},\frac{qc}{a},\frac{qc}{b},\frac{de}{qbc};q)_\infty}{(\frac{q}{a},\frac{q}{b},\frac{qc}{d},\frac{qc}{e},\frac{de}{qb};q)_\infty\vt(\frac{1}{c};q)}\qhyp32{\frac{qc}{d},\frac{qc}{e},\frac{de}{qb}}{\frac{qc}{a},\frac{qc}{b}}{q,q}.
\end{eqnarray}
There will be other specialized transformations of a ${}_3\psi_3$ if one sets $f=q$ in \eqref{firstpsi33}, \eqref{secpsi33}. We leave these transformations to the reader.
\end{rem}
\begin{cor}Let $0<|q|<1$, $a,b,c,d,e,f\in\CCast$ such that $|def|<|qabc|$. Then
\begin{eqnarray}
&&\hspace{-0.7cm}\qpsi33{a,b,c}{d,e,f}{q,
\frac{def}{qabc}}=\frac{\frac{c}{d}(q,c,\frac{q}{d},\frac{q}{e},\frac{e}{a},\frac{e}{b},\frac{qf}{e},\frac{de}{qc},\frac{ef}{qc},\frac{q^2c}{de},\frac{q^2c}{ef};q)_\infty}{(f,\frac{q}{a},\frac{q}{b},\frac{e}{q},\frac{e}{d},\frac{f}{c},\frac{q^2}{e},\frac{qc}{e},\frac{qc}{f},\frac{de}{c},\frac{qc}{de};q)_\infty}\qhyp32{\frac{qa}{e},\frac{qb}{e},\frac{qc}{e}}{\frac{qd}{e},\frac{qf}{e}}{q,\frac{def}{qabc}}\nonumber\\*
&&\hspace{2.7cm}-\frac{\frac{c}{d}(q,c,\frac{q}{d},\frac{q}{e},\frac{d}{a},\frac{d}{b},\frac{qe}{d},\frac{qf}{d},\frac{de}{qc},\frac{df}{qc},\frac{q^2c}{de},\frac{q^2c}{df};q)_\infty}{(f,\frac{q}{a},\frac{q}{b},\frac{d}{q},\frac{e}{d},\frac{f}{c},\frac{q^2}{d},\frac{qc}{d},\frac{qc}{f},\frac{qd}{e},\frac{de}{c},\frac{qc}{de};q)_\infty}\qhyp32{\frac{qa}{d},\frac{qb}{d},\frac{qc}{d}}{\frac{qe}{d},\frac{qf}{d}}{q,\frac{def}{qabc}}\nonumber\\*
&&\hspace{2.7cm}-
\frac{(q,c,\frac{q}{d},\frac{q}{e},\frac{q}{f},\frac{qc}{a},\frac{qc}{b};q)_\infty}
{(qc,\frac{1}{c},\frac{q}{a},\frac{q}{b},\frac{qc}{d},\frac{qc}{e},\frac{qc}{f};q)_\infty}
\qhyp32{\frac{qc}{d},\frac{qc}{e},\frac{qc}{f}}{\frac{qc}{a},\frac{qc}{b}}{q,q}.
\label{3psi3trans}
\end{eqnarray}
\label{cor51}
\end{cor}
\begin{proof}
Start with Theorem~\ref{thm43} and take the limit as $d\to\infty$. The limit of the last term on the right-hand side of \eqref{4psi4eq} vanishes. The limit of the third to last term on the right-hand side of \eqref{4psi4eq} becomes the last term on the right-hand side of \eqref{3psi3trans}. The limits of the two ${}_8W_7$'s become ${}_7W_6^1$'s which are evaluated as special ${}_3\phi_2$'s using \eqref{W761tran}. Using the theta function addition formula \eqref{tfa} 
completes the proof.
\end{proof}

\begin{rem}
One of the referees inquired whether Corollary~\ref{cor51} might be connected with \cite[(5.4.4)]{GaspRah} as a ${}_3\psi_3$ with the required argument. Performing the necessary substitutions leads to the following symmetric sum of three terms, namely:
\begin{eqnarray}
&&\hspace{-1.5cm}\qpsi33{a,b,c}{d,e,f}{q,\frac{def}{qabc}}=\frac{(q;q)_\infty\vartheta(qa,qb,qc;q)}{(d,e,f,\frac{q}{a},\frac{q}{b},\frac{q}{c};q)_\infty\vartheta(\frac{def}{qabc};q)}\nonumber\\
&&\hspace{1.0cm}\times
\II{a;b,c} \frac{a^2(\frac{d}{a},\frac{e}{a},\frac{f}{a},\frac{qa}{b},\frac{qa}{c};q)_\infty\vartheta(\frac{def}{qbc};q)}{(\frac{a}{b},\frac{a}{c},\frac{qb}{a},\frac{qc}{a};q)_\infty\vartheta(qa;q)}
\qhyp32{\frac{qa}{d},\frac{qa}{e},\frac{qa}{f}}{\frac{qa}{b},\frac{qa}{c}}{q,q}.
\label{gs544}
\end{eqnarray}
Corollary~\ref{cor51} and \eqref{gs544} might
very-well be related and there is a lot of freedom to work here because the nonterminating ${}_3\phi_2$'s with non-$q$ argument satisfy the transformations given by \cite[(III.9), (III.10), (III.34)]{GaspRah}. However, we were unable to find a direct route which relates both of these transformation formulas.    
\end{rem}

\begin{rem}
If one takes the limit as either $d\to q$ or $e\to q$ in Corollary~\ref{cor51}, then two of the terms on the right-hand side vanish leading to trivial identities. Alternatively, if one takes the limit as $f\to q$ in Corollary~\ref{cor51}, then one of the terms on the right-hand side vanishes leading to the following transformation formula for a special nonterminating ${}_3\phi_2$, namely
\begin{eqnarray}
&&\hspace{-0.2cm}\qhyp32{a,b,c}{d,e}{q,\frac{de}{abc}}=
\frac{(\frac{q}{d},\frac{q}{e};q)_\infty\vt(\frac{c}{d},\frac{de}{qc};q)}{(\frac{q}{a},\frac{q}{b},\frac{q}{c};q)_\infty\vt{(\frac{qc}{d},\frac{de}{c};q)}}
\nonumber\\*
&&\hspace{0.5cm}\times\left(
\frac{(\frac{d}{a},\frac{d}{b},\frac{d}{c},\frac{qe}{d};q)_\infty}{(\frac{d}{q};q)_\infty\vt(\frac{e}{d};q)}\qhyp32{\frac{qa}{d},\frac{qb}{d},\frac{qc}{d}}{\frac{q^2}{d},\frac{qe}{d}}{q,\frac{de}{abc}}-\frac{(\frac{e}{a},\frac{e}{b};q)_\infty\vt(\frac{e}{c};q)}{(\frac{e}{q},\frac{e}{d},\frac{qc}{e};q)_\infty}\qhyp32{\frac{qa}{e},\frac{qb}{e},\frac{qc}{e}}{\frac{q^2}{e},\frac{qd}{e}}{q,\frac{de}{abc}}\right).
\label{subqf}
\end{eqnarray}
The above transformation formula can be verified by starting with the three-term transformation formula \cite[(III.33)]{GaspRah} which is invariant under the replacement of the variables $d,e$. Performing this variable interchange and solving for the second term on the right-hand side obtains two identities which can then be multiplied by the corresponding infinite $q$-shifted factorials in \eqref{subqf} and then subtracted. The resulting expression then simplifies to the correct result.
\end{rem}

\subsection{Transformations for a ${}_2\psi_3$ series}

By starting with \eqref{cor44eq} and taking advantage of the symmetry in the parameters $b,c,d,e$, one can obtain a transformation of a special ${}_2\psi_3$.
\begin{thm}Let $0<|q|<1$, $a,b,c,d,e\in\CCast$. Then 
\begin{eqnarray}
&&\hspace{-1cm} \qpsi23{a,b}{c,d,e}{q,\frac{cde}{qab}}=\frac{(b,\frac{d}{a},\frac{e}{a},\frac{q^2b}{de},\frac{cde}{qab};q)_\infty}{(c,\frac{q}{a},\frac{qb}{d},\frac{qb}{e},\frac{de}{qa};q)_\infty}\qpsi23{\frac{de}{qa},\frac{de}{qb}}{d,e,\frac{cde}{qab}}{q,c}\nonumber\\*
&&\hspace{0.0cm}
+\frac{(q,b,\frac{q}{e},\frac{qb}{a};q)_\infty\vt(d;q)}{(c,d,\frac{q}{a},\frac{qb}{c},\frac{qb}{d},\frac{qb}{e};q)_\infty}
\left(\frac{\vt(\frac{cde}{qab};q)}{\vt(\frac{qa}{de};q)}\qphyp311{\frac{d}{a},\frac{e}{a},\frac{qb}{c}}{\frac{qb}{a}}{q,q}-\frac{\vt(c;q)}{\vt(\frac{1}{b};q)}\qphyp311{\frac{qb}{c},\frac{qb}{d},\frac{qb}{e}}{\frac{qb}{a}}{q,q}\right)\label{firstpsi23}\\*
&&\hspace{0.0cm}=\frac{(b,\frac{c}{a},\frac{q^2b}{de};q)_\infty}{(c,\frac{qb}{d},\frac{qb}{e};q)_\infty}\qpsi22{a,\frac{de}{qb}}{d,e}{q,\frac{c}{a}}-\frac{(q,b,\frac{q}{c},\frac{q}{d},\frac{q}{e},\frac{qb}{a};q)_\infty}{(\frac{q}{a},\frac{qb}{c},\frac{qb}{d},\frac{qb}{e};q)_\infty\vt(\frac{1}{b};q)}\qphyp311{\frac{qb}{c},\frac{qb}{d},\frac{qb}{e}}{\frac{qb}{a}}{q,q}.
\label{secpsi23}
\end{eqnarray}
\end{thm}
\begin{proof}
Start with \eqref{cor44eq} and compare the two invariant terms with $d$ and $e$ replaced which are equal. Cancel like terms, and solving for the first ${}_2\psi_3$ produces \eqref{firstpsi23}. To obtain \eqref{secpsi23}, start with with \eqref{firstpsi33} and take the limit as $b\to\infty$. In both cases, the second term on the right-hand side vanishes in this limit and the ${}_3\psi_3$ on the right-hand side becomes a ${}_2\psi_2$. This completes the proof. 
\end{proof}

\subsection{Transformations for the ${}_2\psi_2$ series of arbitrary argument}

\begin{cor}
\label{cor61}
Let $0<|q|<1$, $a,b,c,d,z\in\CCast$, such that $|cd/(abz)|<1$, $|z|<1$. Then one has the following 
transformation of a ${}_2\psi_2$ in terms of a sum of a ${}_3\psi_3$ and a ${}_3\phi_2$ with arguments $q$ and with one vanishing denominator element, and a very-well-poised ${}_6\Psi_6^2$, namely 
\begin{eqnarray}
&&\hspace{-0.3cm}\qpsi22{a,b}{c,d}{q,z}=\frac{(\frac{c}{a},\frac{c}{b},\frac{qc}{abz};q)_\infty}{(c,\frac{c}{ab},\frac{cd}{abz};q)_\infty}\qppsi321{a,b,\frac{abz}{c}}{d,\frac{qab}{c}}{q,q}+
\frac{(q,a,b,\frac{cd}{ab},\frac{abz}{c},\frac{qc}{abz};q)_\infty}{(c,d,\frac{ab}{c},\frac{qc}{ab},z,\frac{cd}{abz};q)_\infty}\qphyp311{\frac{c}{a},\frac{c}{b},z}{\frac{cd}{ab}}{q,q}\label{eq:2psi23}\\
&&\hspace{2.4cm}=\frac{(az,bz,\frac{qc}{abz},\frac{qd}{abz};q)_\infty}{(z,abz,\frac{q^2}{abz},\frac{cd}{abz};q)_\infty}\qpWpsi662{\frac{abz}{q}}{a,b,\frac{abz}{c},\frac{abz}{d}}{q,\frac{cdz}{q}}.
\label{transzero}
\end{eqnarray}
\end{cor}
\begin{proof}
Starting with \eqref{psi682phi2} (which is \eqref{transzero}) and equating it with
\eqref{6psi63ph1} while replacing  
$\{a,b,c,d,e\}\mapsto\left\{{abz}/{q},a,b,{qa}/{c},{qa}/{d}\right\}$
completes the proof.
\end{proof}
\begin{rem}
Note that setting $d=q$ in  
\eqref{eq:2psi23}
converts it to  \eqref{2phi3phi2dzero}.
\end{rem}
\begin{rem}
The transformation which describes the parameter interchange symmetry in the above ${}_2\psi_2$ is \eqref{qpsi22trans}.
\end{rem}

\subsection{Transformation for a ${}_1\psi_3$ series}

We can obtain a transformation for a ${}_1\psi_3$ by starting with transformations of a ${}_2\psi_3$ and taking the limit as one of the numerator parameters goes to infinity.
\begin{thm}Let $0<|q|<1$, $a,b,c,d\in\CCast$. Then
\begin{eqnarray}
&&\hspace{-.5cm}\qpsi13{a}{b,c,d}{q,\frac{bcd}{qa}}=
\frac{(a,\frac{q^2a}{bd};q)_\infty}{(c,\frac{qa}{b},\frac{qa}{d};q)_\infty}\qpsi12{\frac{bd}{qa}}{b,d}{q,c}-\frac{(q,a,\frac{q}{b},\frac{q}{c},\frac{q}{d};q)_\infty}{(\frac{qa}{b},\frac{qa}{c},\frac{qa}{d};q)_\infty\vt(\frac{1}{a};q)}\qphyp302{\frac{qa}{b},\frac{qa}{c},\frac{qa}{d}}{-}{q,q}.
\end{eqnarray}
\end{thm}
\begin{proof}
Start with \eqref{firstpsi23} or \eqref{secpsi33} and let $a\to\infty$. In both cases the limit leads to the same transformation, which completes the proof.
\end{proof}

\subsection{Summation for a ${}_0\psi_1$ series}

Taking the limit $b\to 0$ in Corollary~\ref{cor54} converts the ${}_2\psi_2$ to a ${}_1\psi_1$ which is
summable using Ramanujan's ${}_1\psi_1$ summation which produces the following summation result. In fact, the following result is just a confluent limit of Ramanujan's ${}_1\psi_1$ summation \eqref{Rama1psi1}, namely 
replacing $z$ by $z/a$, taking $a\to\infty$ and  replacing $b$ by $a$.

\begin{cor}
\label{cor63b}
Let $0<|q|<1$, $a,b,c,z\in\CCast$. Then
\begin{eqnarray}
&&\hspace{-10.7cm}
\qpsi01{-}{a}{q,z}
=\frac{(q,z,\frac{q}{z};q)_\infty}{(a,\frac{a}{z};q)_\infty}.
\end{eqnarray}
\end{cor}
\begin{proof}
Start with Corollary~\ref{cor54} and take the limit $b\to 0$ which converts the ${}_2\psi_1^1$ to a ${}_1\psi_1$, then replace $c\mapsto b$. The ${}_1\psi_1$ can be summed using
Ramanujan's ${}_1\psi_1$ summation
\eqref{Rama1psi1}. Solving for the ${}_0\psi_1$ replacing $\{az,b\}\mapsto \{z,a\}$ and canceling common factors completes the proof.
\end{proof}

\subsection{Summation for the ${}_0\psi_0^2$ series}

\begin{cor}Let $0<|q|<1$, $z\in\CCast$. Then
\begin{eqnarray}
&&\hspace{-3.0cm}\qppsi002{-}{-}{q,z}
=\frac{1}{(z,\frac{q^2}{z};q)_\infty}\qpWpsi226{\frac{z}{q}}{-}{q,\frac{z^3}{q}}
=\frac{(q^2;q^2)_\infty\vartheta
(\frac{z}{q};q^2)\vartheta(z^2;q^4)}{\vartheta(\frac{z}{q};q)}.
\end{eqnarray}
\end{cor}
\begin{proof}
Start with Corollary~\ref{cor56}
and take the limit as $b\to \infty$. The limit of the ${}_1\psi_0^2$ is the ${}_0\psi_0^2$. 
Note that the triple-product identity
\eqref{tfdef}, given by 
\begin{equation}
(q,z,\tfrac{q}{z};q)_\infty=\sum_{n=-\infty}^\infty q^{\binom{n}{2}}(-z)^n=\qppsi001{-}{-}{q,z}.
\end{equation}
If one replace $q\mapsto q^2$, then we obtain
\begin{equation}
(q^2,z,\tfrac{q^2}{z};q^2)_\infty=\sum_{n=-\infty}^\infty q^{2\binom{n}{2}}(-z)^n=\qppsi001{-}{-}{q^2,z}=\qppsi002{-}{-}{q,-z}.
\end{equation}
This completes the proof.
\end{proof}


\section*{Acknowledgements}
Much appreciation for valuable discussions with S.~Ole Warnaar.
We further wish to thank the reviewers who helped a lot to improve this paper.


\def\cprime{$'$} \def\dbar{\leavevmode\hbox to 0pt{\hskip.2ex \accent"16\hss}d}

\section*{Statements \& Declarations}

\noindent {\bf Funding}\\[0.2cm]
\noindent The second author was partially supported by the Austrian Science Funds FWF, grant 10.55776/P32305.
Apart from that, the authors declare that no funds, grants, or other support were received during the preparation
of this manuscript.\\[0.0cm]

\noindent {\bf Competing Interests}\\[0.2cm]
\noindent The authors have no relevant financial or non-financial interests to disclose.\\[0.0cm]

\noindent {\bf Author Contributions}\\[0.2cm]
\noindent All authors contributed to the study, conception and design. Material preparation, data collection and analysis were performed by all authors. The first draft of the manuscript was written by all authors and all authors commented on previous versions of the manuscript. All authors read and approved the final manuscript.\\[0.0cm]

\end{document}